\documentclass[10pt]{article}

\usepackage[english]{babel}
\usepackage[latin1]{inputenc}
\usepackage[T1]{fontenc}
\usepackage[usenames]{color}
\usepackage{amsmath}
\usepackage{amssymb}
\usepackage{hhline}
\usepackage{multirow}
\usepackage{colortbl}
\usepackage{enumerate}
\usepackage{array}
\usepackage{soul}
\usepackage{graphicx}
\usepackage{eurosym}
\usepackage{mathrsfs}
\usepackage{dsfont}
\usepackage{amsthm}
\usepackage{makecell}
\usepackage[titletoc]{appendix}
\usepackage[titles]{tocloft}
\usepackage{indentfirst}

\newtheorem{theorem}{Theorem}[section]
\newtheorem{proposition}{Proposition}[section]
\newtheorem{remark}{Remark}[section]

\newtheorem{definition}{Definition}[section]

\title{A pseudospectral method for Option Pricing with Transaction Costs under Exponential Utility\thanks{Research supported by Spanish AEI/MINECO
under grant MTM2016-78995-P and by Spanish Junta de Castilla y Le\'{o}n (cofinanced by FEDER funds) under grant VA105G18. The second author acknowledges the support of Junta de Castilla y Le\'{o}n through grant LE103G18.}}

\author{Javier de Frutos\thanks{Instituto de Matem\'{a}ticas (IMUVA), Universidad de Valladolid, Paseo de Bel\'{e}n 7, Valladolid, Spain. e-mail:frutos@mac.uva.es} and V\'{\i}ctor Gat\'{o}n\thanks{Universidad de Leon and Instituto de Matem\'{a}ticas (IMUVA),  Callej\'{o}n Campus Vegazana, s/n, Le\'{o}n, Spain. e-mail:vgatb@unileon.es}}

\begin{document}

\maketitle

\begin{abstract} This paper concerns the design of a Fourier based pseudospectral numerical method for the model of European Option Pricing with transaction costs under Exponential Utility  derived by Davis, Panas and Zariphopoulou in \cite{Davis2}. Computing the option price involves solving two stochastic optimal control problems. With a Exponential Utility function, the dimension of the problem can be reduced, but one has to deal with high absolute values in the objective function. In this paper, we propose two changes of variables that reduce the impact of the exponential growth. We propose a Fourier pseudospectral method to solve the resulting non linear equation. Numerical analysis of the stability, consistency, convergence and localization error of the method are included. Numerical experiments support the theoretical results. The effect of incorporating transaction costs is also studied.

\

\textbf{Keywords:} {Option Pricing, Exponential Utility, Transaction costs, Spectral method.}
\end{abstract}

\section{Introduction}

This paper concerns the design of a pseudospectral numerical method for the model of European Option Pricing with transaction costs under Exponential Utility derived by Davis, Panas and Zariphopoulou in \cite{Davis2}. Let us consider a market form by a risky stock and a riskless bank account (or bond). When transaction costs are considered, the Black-Scholes strategy of a replicating portfolio, \cite{Black}, is unfeasible because it requires a continuous portfolio rebalancing with unbounded costs.

From the point of view of the seller, we can price the Option using a technique referred as ``Indifference Pricing'', \cite{Carmona} or \cite{Davis2}. We define an adequate function (strictly increasing and concave), which allows us to measure the utility of the wealth. For a fixed initial amount of money, we build two scenarios. In the first one, only the stock and the bond are considered and we solve an Optimal Investment problem under transaction costs. In the second one, we receive a certain amount $p_w$ for selling an option and, with the new total amount of money, we solve again the Optimal Investment problem including this time the obligation acquired when selling the option. The quantity $p_w$ that equals the expected terminal utility of both scenarios will be the price of the contract. The technique is also interesting because it reflects the no-linearity of the price in relation with the number of contracts negotiated, in contrast to the Black-Scholes model \cite{Carmona}.

Proportional transaction costs were first introduced in \cite{Magill}. In \cite{Davis2}, authors price European Options with transaction costs under Exponential Utility. This utility function gives tractable equations and it allows to reduce one of the dimensions of the problem, but it may give numerical difficulties in lognormal models due to the growth of the utility function. In the present paper, we propose two changes of variables to reduce the impact of the exponential growth. In spite of being non-linear, the resulting equation can be numerically solved efficiently with a Fourier pseudospectral method.

As it is well known, spectral methods (see \cite{Canuto}), are a class of spatial discretizations for partial differential equations with an order of convergence that depends only on the regularity of the function to be approximated. Several papers (see, for example, \cite{Breton}, \cite{Frutos} or \cite{FrutosGaton1}) have used spectral methods for problems in Finance with good results. For instance,  in  \cite{Chiarella} a Fourier-Hermite procedure to the valuation of american options is presented. In \cite{Breton2} the authors use an adaptive method with Chebyshev polynomials coupled with a dynamic programing procedure for contracts with early exercise features. A spectral procedure coupled with a reduced basis method is used in \cite{FrutosGaton2} to calibrate a high dimensional GARCH model. In \cite{Oosterlee} a very efficient procedure for asian options defined on arithmetic averages has been proposed. In all cases, the spectral-based methods have been proved to be competitive with other alternatives in terms of precision versus computing time needed to compute the numerical solution.

Theoretical results analyzing stability, consistency, convergence and localization error of the pseudospectral method are included. When transaction costs disappear, and all risks become again hedgeable, the replication price, i.e., the Black-Scholes price, must be recovered, \cite{Carmona}, \cite{Davis2}. We use this fact to check the precision and efficiency of the pseudospectral method.

The outline of the paper is as follows. In Section \ref{OPTCEUtM} a description of the model as it can be found in \cite{Davis2} is presented. In Section \ref{Ch4intrbankstate}, the problem is equivalently reformulated for technical reasons. Section \ref{OPTCEUtPm} is devoted to the two changes of variables and the development of a Fourier pseudospectral method to solve the new non-linear partial differential equation. A theoretical analysis of the method is included. Section \ref{OPTCEUmda} is devoted to the numerical analysis. The effect of incorporating transaction costs will also be studied. In order to not overload the paper, the proof of all the theoretical results are included in the appendix.

\section{The model}\label{OPTCEUtM}

We consider the European Option pricing problem with transaction costs \cite{Davis2}. Let $(\Omega,\mathcal{F},P)$ be a filtered probability space. Let us consider an investor who holds amount $\bar{X}(t)$ in the bank account and $\bar{y}(t)$ shares of a certain stock $\bar{S}(t)$. The dynamics of the processes is
\begin{equation}\label{Ch4modeldynamecu}
\left\{
\begin{aligned}
d\bar{X}(t)&=r\bar{X}(t)dt-(1+\lambda)\bar{S}(t)dL(t)+(1-\mu)\bar{S}(t)dM(t), \\
d\bar{y}(t)&=dL(t)-dM(t), \\
d\bar{S}(t)&=\bar{S}(t)\alpha dt+\bar{S}(t)\sigma dz_t,
\end{aligned}
\right.
\end{equation}
where $r$ denotes the constant risk-free rate, $\alpha$ is the constant expected rate of return of the stock, $\sigma$ > 0 is the constant volatility of the stock, and $z_t$ is a is a standard Brownian motion such that $\mathscr{F}^{z}_t\subseteq\mathscr{F}$ where $\mathscr{F}^{z}_t$ is the natural filtration induced by $z_t$. We suppose that $L(t)$ and $M(t)$ are adapted, right-continuous, nonnegative and nondecreasing processes representing the cumulative number of shares bought and sold respectively. $\lambda\ge 0$ and $0\le\mu<1$, represent the constant proportional transaction costs incurred on the purchase or sale of the stock.

The investor may borrow from the bank at interest rate $r$ and $y\in \mathbb{R}$, so long and short positions are both accepted. The liquidated cash value of a portfolio, denoted by $c(y,S)$, is given by:
\begin{equation}
c(y,S)=(1-\mu)Sy, \quad \text{if} \ y\geq0,
\end{equation}
if the investor is long in the stock or
\begin{equation}
c(y,S)=(1+\lambda)Sy, \quad \text{if} \ y<0,
\end{equation}
in case the investor is short. Let $T$ be a fixed maturity, when our investor has to liquidate his portforlio. We consider \cite{Davis2}, two different scenarios.

In \textbf{Scenario $j=1$}, the investor holds money in the bank account and in shares, but he/she has not sold an option. At maturity, the net wealth of the investor $W_1(T)$ is given by:
\begin{equation}\label{Ch4problnumcontratos1}
W_1(T)=\bar{X}(T)+c(\bar{y}(T),\bar{S}(T)).
\end{equation}

In \textbf{Scenario $j=w$}, prior to enter into the market, the investor has sold an European Option with strike $K$ and maturity $T$. The net wealth of the investor at maturity, $W_w(T)$, is:
\begin{equation}\label{Ch4problnumcontratos2}
W_w(T)=
\left\{
\begin{aligned}
& \bar{X}(T)+c(\bar{y}(T),\bar{S}(T)),  &&  \text{if} \ \bar{S}(T)<K, \\
& \bar{X}(T)+ K +c(\bar{y}(T)-1,\bar{S}(T)), &&  \text{if} \ \bar{S}(T)\geq K,
\end{aligned}
\right.
\end{equation}
which corresponds to the net value of the portfolio if the option is not exercised (respectively the net value minus one share plus the strike value if the option is exercised).

Given an election of an utility function $U(w)$, that is a continuous, strictly increasing and concave function, and for a position $(\bar{X}(t),\bar{y}(t))=(X,y)$, the optimal value function  is given by:
\begin{equation}\label{Ch4defvaluefunfor}
V_j(t,X,y,S)=\underset{\pi\in\tau(X,y)}{\sup}\mathbb{E}\left\{\left.U(W_j(T)\right|\left(\bar{X}(t),\bar{y}(t),\bar{S}(t)\right)=\left(X,y,S\right)\right\},
\end{equation}
where $(t,X,y,S)\in [0,T]\times \mathbb{R}\times \mathbb{R}\times \mathbb{R^{+}}$ and $j\in\{1,w\}$. From now on, we assume \cite{Davis2}, that $U(w)$ is the exponential utility function
\begin{equation}\label{Ch4utilfun}
U(x)=1-\exp(-\gamma x).
\end{equation}
for some $\gamma>0$ and where we note that $\gamma=-\frac{U''(x)}{U'(x)}$, the index of risk aversion, is independent of the investor's wealth. Set $\tau(X,y)$ corresponds to the set of admissible trading strategies and it is defined in Subsection \ref{OPTCEUtMExandun}, where we discuss the existence and uniqueness of a solution.

The Optimal Investment problems $V_j, \ j\in\{1,w\}$ can be solved for any initial position but, when we want to price an option, for simplicity we assume, \cite{Davis2}, that prior to enter into the market, the position of an investor is always a certain amount of money in the bank account $\bar{X}(t^-)=X$ and no holdings in the stock $\bar{y}(t^{-})=0$.

The indifferent price $p_w(X,t,S)$ of one European Option for an investor with an initial position $(\bar{X}(t^{-}),\bar{y}(t^{-}))=(X,0)$ is the price which leaves him indifferent between not selling an option ($j=1$) or selling one option ($j=w$) for an amount $p_w(X,t,S)$, i.e. the quantity which equals
\begin{equation}\label{Ch4indftpr}
V_1(t,X,0,S)=V_w\left(t,X+p_w(X,t,S),0,S\right).
\end{equation}

\begin{remark}
\normalfont{
The investor may sell one or $n$ European Options. The indifferent price is not linear in the number of contracts (see \cite{Carmona}) but, for simplicity, the problem is solved for just one contract.

The development is identical substituting 1 by $n$ in formula (\ref{Ch4problnumcontratos2}).}
\end{remark}

After obtaining the Hamilton-Jacobi-Bellman equations (see \cite{Davis2}) associated with the two stochastic control problems $j\in\{1,w\}$, the results suggest that the optimization problem is a free boundary problem given by
\begin{equation}\label{Ch4ecupdeDaviscomple}
\begin{aligned}
\max & \left\{\frac{\partial V_j}{\partial y}-(1+\lambda)S\frac{\partial V_j}{\partial X},-\left(\frac{\partial V_j}{\partial y}-(1-\mu)S\frac{\partial V_j}{\partial X}\right),\right. \\
& \left. \frac{\partial V_j}{\partial t}+rX\frac{\partial V_j}{\partial X}+\alpha S\frac{\partial V_j}{\partial S}+\frac{1}{2}\sigma^2 S^2 \frac{\partial V_j}{\partial S^2}\right\}=0,
\end{aligned}
\end{equation}
subject to
\begin{equation}
V_j(T,X,y,S)=U\left(W_j(T)\right),
\end{equation}
where $(t,X,y,S)\in [0,T]\times \mathbb{R}\times \mathbb{R}\times \mathbb{R^{+}}$. The existence and uniqueness of a solution is discussed in Subsection \ref{OPTCEUtMExandun}.

Under the Exponential Utility, it can be proved (see \cite{Davis2}) that the value function given by (\ref{Ch4defvaluefunfor}) can be rewritten as:
\begin{equation}\label{Ch4obtencQ}
V_j(t,X,y,S)=1-\exp\left(-\gamma\frac{X}{\delta(T,t)}\right)Q_j(t,y,S),
\end{equation}
where $Q_j(t,y,S)$ is a convex nonincreasing continuous function in $y$ and $S$ given by
\begin{equation}\label{relorigyqdavis}
Q_j(t,y,S)=1-V_j(t,0,y,S).
\end{equation}

This result has a very important interpretation: ``The amount invested in the risky asset is independent of the total wealth.''

The indifferent price $p_w(X,t,S)$ given by (\ref{Ch4indftpr}) can be explicitly computed with (\ref{Ch4obtencQ}) and is given by
\begin{equation}\label{Ch4indffpricenocosts}
p_w(X,t,S)=\frac{\delta(T,t)}{\gamma}\log\left(\frac{Q_w(t,0,S)}{Q_1(t,0,S)}\right).
\end{equation}
where note that it is independent of the initial wealth $p_w(X,t,S)=p_w(t,S)$.

Substituting (\ref{Ch4obtencQ}) into the partial differential equation (\ref{Ch4ecupdeDaviscomple}), we obtain:
\begin{equation}\label{Ch4ecupdedefin}
\begin{aligned}
\min & \left\{\frac{\partial Q_j}{\partial y}+\frac{\gamma(1+\lambda)S}{\delta(T,t)}Q_j,-\left(\frac{\partial Q_j}{\partial y}+\frac{\gamma(1-\mu)S}{\delta(T,t)}Q_j\right), \right. \\
& \left. \frac{\partial Q_j}{\partial t}+\alpha S\frac{\partial Q_j}{\partial S}+\frac{1}{2}\sigma^2 S^2 \frac{\partial Q_j}{\partial S^2}\right\}=0,
\end{aligned}
\end{equation}
defined in $[0,T]\times\mathbb{R}\times \mathbb{R^{+}}$. The terminal conditions are given by:
\begin{equation}\label{Ch4ecupdedefinmatcon1}
Q_1(T,y,S)=\exp(-\gamma c(y,S)),
\end{equation}
and
\begin{equation}\label{Ch4ecupdedefinmatcon2}
Q_w(T,y,S)=\exp\left(-\gamma \left(I_{(S< K)}c(y,S)+I_{(S\geq K)}\left[c(y-1,S)+K\right]\right)\right).
\end{equation}

We conjecture, as in \cite{Davis2}, that the space is divided by (\ref{Ch4ecupdedefin}) in three regions:
\noindent \textbf{1.} The Buying Region (BR), where the value function satisfies
\begin{equation}\label{Ch4ecdeffBR}
\frac{\partial Q_j}{\partial y}+\frac{\gamma(1+\lambda)S}{\delta(T,t)}Q_j=0,
\end{equation}
\noindent \textbf{2.} The Selling Region (SR), where the value function satisfies
\begin{equation}\label{Ch4ecdeffSR}
-\left(\frac{\partial Q_j}{\partial y}+\frac{\gamma(1-\mu)S}{\delta(T,t)}Q_j\right)=0,
\end{equation}
\noindent \textbf{3.} The No Transactions Region (NT), where the value function is the solution of the following partial differential equation:
\begin{equation}\label{Ch4ecdeffNT}
\frac{\partial Q_j}{\partial t}+\alpha S\frac{\partial Q_j}{\partial S}+\frac{1}{2}\sigma^2 S^2 \frac{\partial Q_j}{\partial S^2}=0.
\end{equation}

The Buying and Selling regions do not intersect, since it is not optimal to buy and sell shares at the same time, laying the No Transactions region between them. The Buying (resp. Selling) frontier is denoted by $y^{\mathscr{B}}_j(t,S)$ (resp. $y^{\mathscr{S}}_j(t,S)$), $j\in\{1,w\}$.

If we are located inside the Buying (resp. Selling) Region, the optimal trading strategy is to immediately buy (resp. sell) shares until reaching the Buying (resp. Selling) frontier.

If the Buying $y^{\mathscr{B}}_j(t,S)$ and Selling $y^{\mathscr{S}}_j(t,S)$ frontiers are known, we can compute the value function  $Q_j(t,y,S), \ j\in\{1,w\}$ explicitly by a simple integration of equations (\ref{Ch4ecdeffBR}) and (\ref{Ch4ecdeffSR}) respectively. If $y\leq y^{\mathscr{B}}_j(t,S)$,
\begin{equation}\label{Ch4ecuzonacompra}
Q_j(t,y,S)=Q_j(t,y^{\mathscr{B}}_j(t,S),S)\exp\left(-\frac{\gamma(1+\lambda)S}{\delta(T,t)}(y-y^{\mathscr{B}}_j(t,S))\right),
\end{equation}
and if $y\geq y^{\mathscr{S}}_j(t,S)$
\begin{equation}\label{Ch4ecuzonaventa}
Q_j(t,y,S)=Q_j(t,y^{\mathscr{S}}_j(t,S),S)\exp\left(\frac{\gamma(1-\mu)S}{\delta(T,t)}(y^{\mathscr{S}}_j(t,S)-y)\right).
\end{equation}
where note that $Q_j(t,y,S), \ j\in\{1,w\}$ is determined in BR (resp. SR) upon the knowledge of $Q_j(t,y^{\mathscr{B}}_j(t,S),S)$ (resp. $Q_j(t,y^{\mathscr{S}}_j(t,S),S)$).

\subsection{Existence and uniqueness of a viscosity solution}\label{OPTCEUtMExandun}

Let $t\in[t_0,T]$. The set of admissible strategies $\tau_E(X_{t_0},y_{t_0})$ consists of the two dimensional, right-continuous, measurable processes $(X^{\pi}(t),y^{\pi}(t))$ which are the solution of (\ref{Ch4modeldynamecu}), corresponding to some pair of right-continuous, measurable $\mathcal{F}_t$-adapted, increasing processes $(L(t),M(t))$ such that
\begin{equation*}
\left\{
\begin{aligned}
& \bar{X}(t_0^{-})=X_{t_0}, \ \bar{y}(t_0^{-})=y_{t_0}, \\
& (X^{\pi}(t),y^{\pi}(t),\bar{S}(t))\in \mathscr{E}_{E}, \quad \forall t\in[t_0,T]
\end{aligned}
\right.
\end{equation*}
where  $E>0$ is a constant which may depend on the policy $\pi$ and
\begin{equation}\label{constraint}
\mathscr{E}_{E}=\left\{(X,y,S)\in \mathbb{R} \times \mathbb{R} \times \mathbb{R}^{+}: \left(x+c(y-1,S)\right)e^{r(T-t)}>-E, \ t\in[t_0,T]\right\},
\end{equation}
where $X$, $y$ and $S$ respectively denote the money in the bank account, the number of shares and the stock price. By convention, $L(t_0^{-})=M(t_0^{-})=0$ but $L(t_0)$ or $M(t_0)$ may be positive.

\begin{remark}
\normalfont{
$\mathscr{E}_{E}$ was originally defined as $\mathscr{E}^{*}_{E}$ by \cite[(4.6)]{Davis2}. Our definition does not alter the results from \cite{Davis2}, but we have been a bit more restrictive ($\mathscr{E}_{E}\subset\mathscr{E}^{*}_{E}$), just to ensure that the trading strategies
\begin{equation*}
\left\{
\begin{aligned}
& y^{\pi}(t)\equiv 0, \quad \text{(no option was sold)}, \\
& y^{\pi}(t)\equiv 1, \quad \text{(one option was sold)},
\end{aligned}
\right.
\end{equation*}
are both admissible for any initial position $(t,X,y,S)\in[t_0,T]\times\mathscr{E}_{E}$. }
\end{remark}

For $(t,X,y,S)\in [0,T]\times{\mathscr{E}_{E}}$, we define the value function as:
\begin{equation}\label{Ch4defvaluefunresdom}
V^{\mathscr{E}_{E}}_j(t,X,y,S)=\underset{\pi\in\tau_E(X,y)}{\sup}\mathbb{E}\left\{\left.U(W_j(T)\right|\left(\bar{X}(t),\bar{y}(t),\bar{S}(t)\right)=\left(X,y,S\right)\right\}.
\end{equation}

In \cite{Davis2} it is proved that for $(t,X,y,S)\in [0,T]\times{\mathscr{E}_{E}}$,  (\ref{Ch4defvaluefunresdom}) is the unique viscosity solution of $(\ref{Ch4ecupdeDaviscomple})$.

We assume, as in \cite{Davis2}, that fixed an initial position $(t_0,X_0,y_0,S_0)$, the value of $V^{\mathscr{E}_{E}}_j(t_0,X_0,y_0,S_0)$ does not depend on the particular choice of $E$ for $E\geq E_0$ big enough. This means that, although the set of allowed trading strategies increases with the value of $E$, we obtain the same result. In \cite{Davis2} it was argued that this occurred because constraint $\mathscr{E}_{E}$ only ruled out suboptimal trading strategies. This may be a consequence of the particular choice of Exponential Utility, which makes strategies wealth-independent (something that can be explicitly checked when there are no transaction costs).

Furthermore, it can be proved that for $(t_0,X_0,y_0,S_0)$ fixed, the value of $V^{\mathscr{E}_{E}}_j(t_0,X_0,y_0,S_0)$ is an increasing but bounded function of $E$. Without entering in technical details, we sketch the idea of the proof. It is an increasing function since the set of allowed trading strategies increases with the value of $E$, so the result will be equal or better. It is a bounded function because the value function of the no transaction costs model (which is explicitly computable and finite) is always an upper bound.

This result, and the numerical experiments, strongly suggest that the assumption made in \cite{Davis2} is correct.

Under this assumption, for $E\geq E_0$ big enough, $V_j(t,X,y,S)$ can be unambiguously defined for any $(t,X,y,S)\in[t_0,T]\times \mathbb{R}\times \mathbb{R}\times \mathbb{R}^{+}$. Based on this, for simplicity in the numerical scheme, we drop the dependance on ${\mathscr{E}_{E}}$ in definition (\ref{Ch4defvaluefunfor}), although for the theoretical results we need to employ (\ref{Ch4defvaluefunresdom}).

\section{Restatement of the problem: Bankruptcy state}\label{Ch4intrbankstate}

In order to analyze the localization error of the pseudospectral method that we are going to propose, we need functions $V^{\mathscr{E}_{E}}_j, \ j\in\{1,w\}$ to be defined in $[0,T]\times\mathbb{R}\times\mathbb{R}\times\mathbb{R}^{+}$. In order to achieve this, we restate the problem, but in a way which preserves the original development.

When an European option is signed (or other derivative), the market (Clearing House), acts as a central counterparty which mediates between the seller and the buyer of the option. The Clearing House checks if the seller of the option can afford all the potential loses that he might have incurred between $[0,t]$, even if the European option cannot be exercised prior to time $T$. Furthermore, if the seller has gone into theoretical bankruptcy at any time $t\in[0,T]$, the Clearing House can confiscate his goods and expel him from the market (see, for example, \cite{BME}).

Simplifying the situation, constraint ${\mathscr{E}_{E}}$ could be understood as a bankruptcy constraint. We allow any trading strategy to the seller of the option but, if at any time $t$ his strategy has led him out outside ${\mathscr{E}_{E}}$, he is automatically expelled from the market, not allowing him to return, and he remains with a residual bankruptcy utility forever. Retaining the previous definitions, we introduce two new value functions.

Let $E>0$, $t\in[0,T]$ and $j\in\{1,w\}$. The value functions are given by
\begin{equation}\label{Ch4defvaluefunbank}
V^{B_E}_j(t,X,y,S)= \underset{\pi\in\tau(X,y)}{\sup}\mathbb{E}\left\{\left.U(W_j(T)\right|\left(\bar{X}(t),\bar{y}(t),\bar{S}(t)\right)=(X,y,S)\right\}, \end{equation}
if $(X,y,S)\in {\mathscr{E}_{E}}$ and by
\begin{equation}\label{Ch4defvaluefunbank2}
V^{B_E}_j(t,X,y,S)=1-\exp(\gamma E),
\end{equation}
otherwise. Set $\tau(X,y)$ denotes that we allow any trading strategy. These new value functions are defined in $[0,T]\times\mathbb{R}\times\mathbb{R}\times\mathbb{R}^{+}$ and they do not alter the model thanks to the following result (the proofs are in the appendix).

\begin{proposition}\label{Ch4equivalencia}
If $(t,X,y,S)\in[0,T]\times \mathscr{E}_{E}$, it holds that
\begin{equation*}
V^{B_E}_j(t,X,y,S)=V^{\mathscr{E}_{E}}_j(t,X,y,S).
\end{equation*}
\end{proposition}

Thanks to Proposition \ref{Ch4equivalencia}, we inherit all the existence and uniqueness results of the original development of the model in \cite{Davis2}. We mention that the state space which corresponds to (\ref{Ch4defvaluefunbank})-(\ref{Ch4defvaluefunbank2}) is divided in four regions, not in three as in \cite{Davis2}. The forth state corresponds to the bankruptcy state but, since the investor has been expelled from the market, in this region no trading strategy has to be obtained.

Similar to the model presented in \cite{Davis2}, we are interested in the limit value of the functions when $E\rightarrow \infty$. Again, thanks to Proposition \ref{Ch4equivalencia}, we can make the same assumption as before, i.e. that the value of the objective functions does not depend of $E$ for $E>E_0$ big enough.

Since the option price is independent of the initial wealth, we will work numerically with a function $Q_j(t,y,S)$ derived of formula (\ref{Ch4obtencQ}) from $V_j=V^{\mathscr{E}_{E}}_j=V^{B_E}_j, \ j \in\{1,w\}$, when $E$ is considered big enough. Let us fix $X=X_0$. We apply formula (\ref{Ch4obtencQ}) to functions $V^{B_E}_j$ in order to obtain functions that we will denote by  $Q^{B_{E,X_0}}_j$. It is clear that $Q_j=Q^{\mathscr{E}_{E,X_0}}_j=Q^{B_{E,X_0}}_j, \ j\in\{1,w\}$ when $E$ is considered big enough. The following result will be employed in the analysis of the localization error in Subsection \ref{OPTCEUCSCE}.

\begin{proposition}\label{Ch4Qbounded}
For $X=X_0$ and $E=E_0$ fixed, it exists $M=M(X_0,E_0)\geq 0$ such that $\forall (t,y,S),\in[0,T]\times\mathbb{R}\times\mathbb{R}^{+}$ it holds
\begin{equation*}
0< Q^{B_{E_0,X_0}}_j \leq M, \quad j\in\{1,w\}.
\end{equation*}
\end{proposition}

\section{Numerical Method}\label{OPTCEUtPm}

The procedure is as follows: First, we perform two changes of variables and compute the corresponding equations. The second step is the localization of the problem. We fix a finite domain and perform an odd-even extension, imposing periodic boundary conditions. Finally, we propose a Fourier Pseudospectral method to solve the partial differential equation. All the steps are summarized in the numerical algorithm in Subsection $\ref{OPTCEUtPmDPNV}$. For finishing, we include a theoretical analysis of the stability and convergence of the pseudospectral method as well as an analysis of the localization error.

\subsection{Change of variables.}

First, we change the stock price to logarithmic scale.
\begin{equation}\label{Ch4logscale}
\hat{x}=\log(S).
\end{equation}
and then consider a new function $H_j(t,y,\hat{x})$ defined by:
\begin{equation}\label{Ch4logfunc}
H_j(t,y,\hat{x})=\log \left(Q_j(t,y,\hat{x})\right), \quad j\in\{1,w\},
\end{equation}
which is admissible after Proposition \ref{Ch4Qbounded}.

In the Buying region, $y\leq y^{\mathscr{B}}_j(t,\hat{x})$, equation (\ref{Ch4ecdeffBR}) becomes
\begin{equation}\label{Ch4ecuzonacompranolin}
H_j(t,y,\hat{x}) =H_j(t,y^{\mathscr{B}}_j(t,\hat{x}),\hat{x})+\left(-\frac{\gamma(1+\lambda)\exp(\hat{x})}{\delta(T,t)}(y-y^{\mathscr{B}}_j(t,\hat{x}))\right),
\end{equation}
and in the Selling region,  $y\geq y^{\mathscr{S}}_j(t,\hat{x})$, equation (\ref{Ch4ecdeffSR}) becomes
\begin{equation}\label{Ch4ecuzonaventanolin}
H_j(t,y,\hat{x}) =H_j(t,y^{\mathscr{S}}_j(t,\hat{x}),\hat{x})+\left(\frac{\gamma(1-\mu)\exp(\hat{x})}{\delta(T,t)}(y^{\mathscr{S}}_j(t,\hat{x})-y)\right),
\end{equation}

Equation (\ref{Ch4ecdeffNT}), which corresponds to not performing transactions, has to be numerically solved and is given by
\begin{equation}\label{Ch4ecunolinealdef}
\frac{\partial H_j}{\partial t}+\left(\alpha-\frac{\sigma^2}{2}\right)\frac{\partial H_j}{\partial \hat{x}}+\frac{1}{2}\sigma^2  \frac{\partial^2 H_j}{\partial \hat{x}^2}+\frac{1}{2}\sigma^2  \left(\frac{\partial H_j}{\partial \hat{x}}\right)^2=0, \quad j\in\{1,w\},
\end{equation}

The value function at maturity is given by $H_1(T,y,\hat{x})=$
\begin{equation}\label{Ch4pde2camfin1}
-\gamma c(y,\exp(\hat{x})),
\end{equation}
and $H_w(T,y,\hat{x})=$
\begin{equation}\label{Ch4pde2camfin2}
-\gamma \left(I_{(\exp(\hat{x})< K)}c(y,\exp(\hat{x}))+I_{(\exp(\hat{x})\geq K)}\left[c(y-1,\exp(\hat{x}))+K\right]\right).
\end{equation}

\begin{figure}[h]
\centering
\includegraphics[width=12cm,height=5 cm]{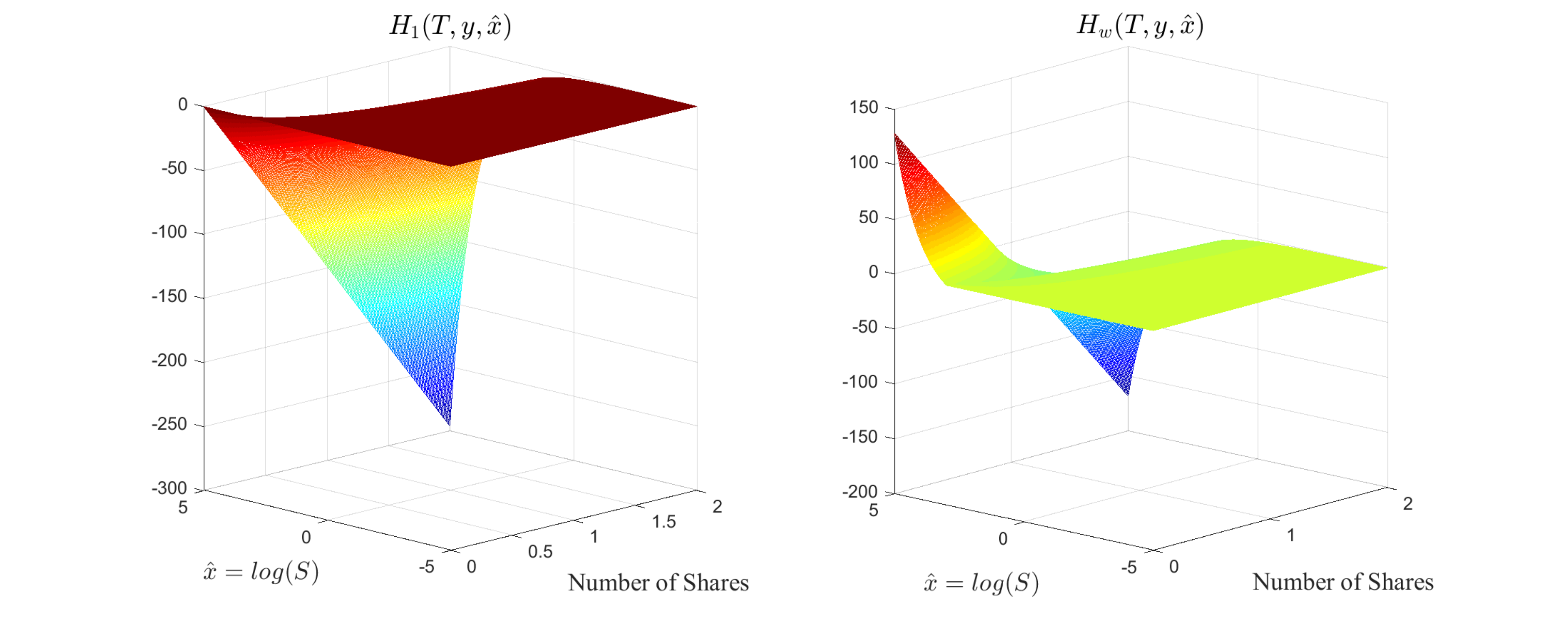}
\caption[$H_j$ terminal values for $\lambda=\mu=0.002$ and $\log(\text{Strike})=3$.]{\label{Ch4valortermhs} Graph of $H_1(T,y,\hat{x})$ (left) and  $H_w(T,y,\hat{x})$ (right),  $\hat{x}\in[-5,5]$, $y\in[0,2]$, $\lambda=\mu=0.002$, $\gamma=1$, $\log(\text{Strike})=3$.}
\end{figure}

We remark that function $H_w(T,y,\hat{x})$ takes much smaller values (absolute value) than function $Q_w(T,y,x)=\exp(H_w(T,y,\exp(\hat{x})))$. In Figure \ref{Ch4valortermhs} we plot the values of function $H_1(T,y,\hat{x})$ (left) and function $H_w(T,y,\hat{x})$ (right) for $\hat{x}\in[-5,5]$, $y\in[0,2]$, $\lambda=\mu=0.002$, $\gamma=1$ and $\log(\text{Strike})=3$.

\subsection{Localization of the problem}\label{OPTCEUtPmnlpde}

The localization procedure of the problem is similar to the one in \cite{Breton}.

We denote by $[L_{\min},L_{\max}]\subset \mathbb{R}$ the \textit{approximation domain}, which is a finite interval large enough to cover the relevant logarithmic stock prices.

We denote by $[\hat{x}_{\min},\hat{x}_{\max}]\subset \mathbb{R}$ the \textit{computational domain}, which is a finite interval such that $L_{\min}>\hat{x}_{\min}$ and $L_{\max}<\hat{x}_{\max}$. The convergence in $[L_{\min},L_{\max}]\subset \mathbb{R}$ of computed prices are obtained by taking $\hat{x}_{\min}\rightarrow-\infty$ and $\hat{x}_{\max}\rightarrow\infty$.

We define the intervals
\begin{equation}\label{Ch4definterv}
\begin{aligned}
& I_1=[\hat{x}_{\min},\hat{x}_{\max}], && I_2=[\hat{x}_{\max},2\hat{x}_{\max}-\hat{x}_{\min}],\\
& I_3=[2\hat{x}_{\max}-\hat{x}_{\min},4\hat{x}_{\max}-3\hat{x}_{\min}], && I=[\hat{x}_{\min},4\hat{x}_{\max}-3\hat{x}_{\min}],
\end{aligned}
\end{equation}
where we note that $I=I_1\cup I_2\cup I_3$.

We define function $H^e_j(t,y,\hat{x}), \ j\in\{1,w\}$ as the odd-even extension oh $H_j$. More precisely,
\begin{equation}\label{Ch4extenfor1}
H^e_j(t,y,\hat{x})=\left\{\begin{aligned}
& H_j(t,y,\hat{x}), \ \text{if} \  \hat{x}\in I_1, \\
& 2H_j(t,y,\hat{x}_{\max})-H_j(t,y,2\hat{x}_{\max}-\hat{x}), \ \text{if} \  \hat{x} \in I_2, \\
& H^e_j(t,y,(4\hat{x}_{\max}-2\hat{x}_{\min})-\hat{x}), \ \text{if} \  \hat{x} \in I_3, \\
& H^e_j(t,y,z), \ \text{if} \ \hat{x}\notin I.
\end{aligned}\right.
\end{equation}
where $z=-\hat{x}+\hat{x}_{\min}+k(4\hat{x}_{\max}-4\hat{x}_{\min})$ and $k\in\mathbb{Z}$ is such that $z\in[\hat{x}_{\min},4\hat{x}_{\max}-3\hat{x}_{\min}]$.
\begin{figure}[h]
\centering
\includegraphics[width=12cm,height=5 cm]{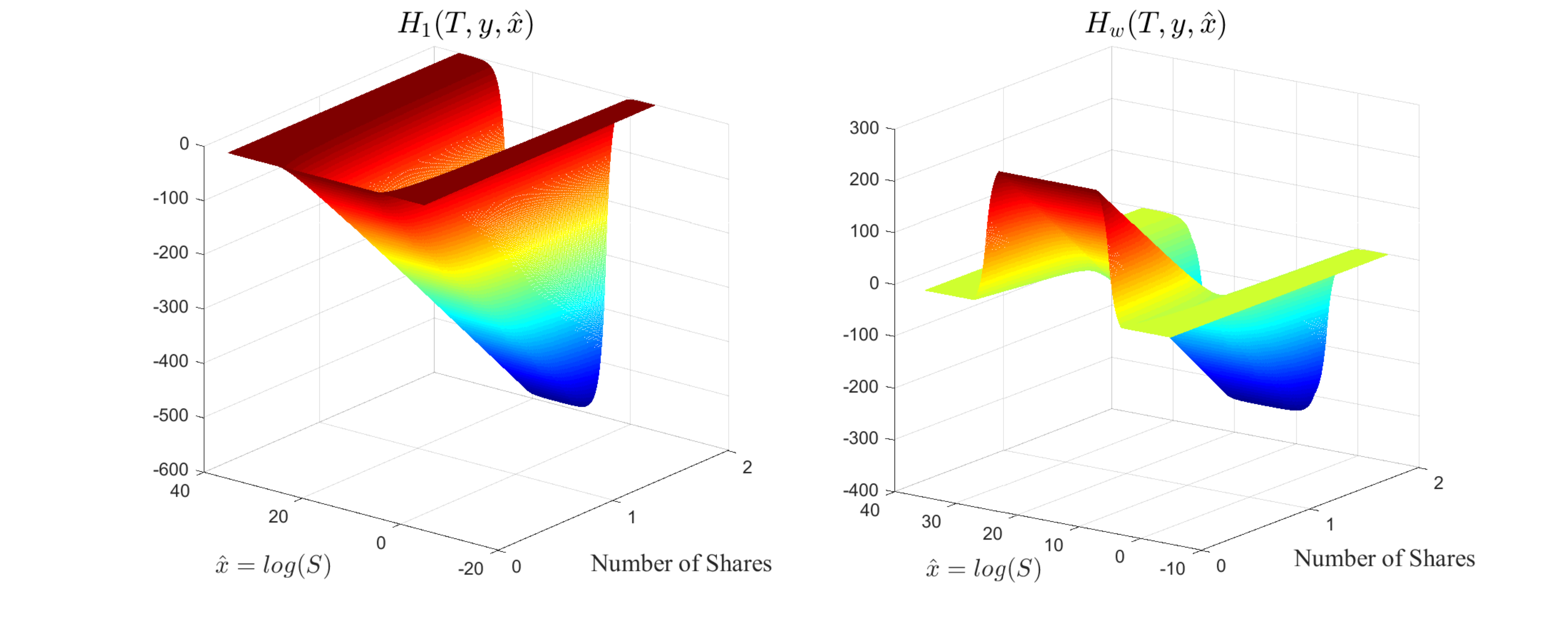}
\caption[$H^e_j$ terminal values for $\lambda=\mu=0.002$ and $\log(\text{Strike})=3$.]{\label{Ch4valortermhsext2} Graph of $H^e_1(T,y,\hat{x})$ (left) and $H^e_w(T,y,\hat{x})$ (right), $\hat{x}\in[-5,35]$, $y\in[0,2]$, $\lambda=\mu=0.002$, $\gamma=1$,  $\log(\text{Strike})=3$.}
\end{figure}

In Figure \ref{Ch4valortermhsext2} we plot function $H^e_1(T,y,\hat{x})$ (left) and function $H^e_w(T,y,\hat{x})$ (right) for $\hat{x}\in[-5,35]$, $y\in[0,2]$, $\lambda=\mu=0.002$, $\gamma=1$ and $\log(\text{Strike})=3$. Functions $H^e_j(T,y,\hat{x}), \ j\in\{1,w\}$ correspond to those of Figure \ref{Ch4valortermhs} after the odd-even extension defined by (\ref{Ch4extenfor1}).

The truncation of the domain that we have proposed induces the so called localization error, due to the extension of the function and the imposition of periodic boundary conditions, since the original function is not periodic. In Subsection \ref{OPTCEUCSCE}, we will prove that the localization error can be made arbitrary small in a fixed approximation domain $[L_{\min},L_{\max}]$ taking the computational domain large enough.

Although periodic boundary conditions can be directly imposed, in order to avoid the Gibbs effect, we have performed first an odd-even extension and then imposed periodic conditions.

Fix a grid $\bar{t}=\{t_m\}_{m=0}^{N}, \ 0=t_0<...<t_{m}<t_{m+1}<...<t_{N}={T}$.

For $t\in[t_{m},t_{m+1}]$, $\hat{x}\in[\hat{x}_{\min},4\hat{x}_{\max}-3\hat{x}_{\min}]$, we define an approximate function $H^p_j, \ j\in\{1,w\}$ as the solution of equation (\ref{Ch4ecunolinealdef}) supplemented with periodic boundary conditions:
\begin{equation*}
\begin{aligned}
& H^p_j(t,y,\hat{x}_{\min})=H^p_j(t,y,4\hat{x}_{\max}-3\hat{x}_{\min}), \\
& \frac{\partial H^p_j}{\partial x}(t,y,\hat{x}_{\min})=\frac{\partial H^p_j}{\partial x}(t,y,4\hat{x}_{\max}-3\hat{x}_{\min}),
\end{aligned}
\end{equation*}
and with the final condition $H^p_j(t_{m+1},y,\hat{x})=H^e_j(t_{m+1},y,\hat{x})$.

The value function $H_j(t_{m+1},y,\hat{x}), \ \hat{x}\in[x_{\min},x_{\max}]$ employed in (\ref{Ch4extenfor1}) is substituted by an approximation computed in the previous step of the numerical procedure,(see Subsection \ref{OPTCEUtPmDPNV}). Finally, and for notational convenience, we change the spatial domain to $x\in[0, 2\pi]$ defining:
\begin{equation}\label{Ch4ecunolinealdeftercam}
u_j(t,y,x)=H^p_j\left(t,y,\hat{x}_{\min}+\frac{4\hat{x}_{\max}-4\hat{x}_{\min}}{2\pi}x\right).
\end{equation}

Therefore, equation (\ref{Ch4ecunolinealdef}) becomes
\begin{equation}\label{Ch4ecunolinealdeftercercam}
\frac{\partial u_j}{\partial t}+A\frac{\partial u_j}{\partial x}+B \frac{\partial^2 u_j}{\partial x^2}+C\left(\frac{\partial u_j}{\partial x}\right)^2=0, \quad j\in\{1,w\},
\end{equation}
supplemented with periodic boundary conditions $u(0,t)=u(2\pi,t), \ u_{x}(0,t)=u_{x}(2\pi,t)$ and where
\begin{equation}\label{Ch4ecunolinealdeftercercamcons}
A=\left(\frac{2\pi}{4\hat{x}_{\max}-4\hat{x}_{\min}}\right)\left(\alpha-\frac{\sigma^2}{2}\right), \quad B=C=\left(\frac{2\pi}{4\hat{x}_{\max}-4\hat{x}_{\min}}\right)^2\frac{1}{2}\sigma^2.
\end{equation}

\subsection{A Pseudospectral method.}

For $N\in\mathbb{N}$, let $S_N$ be the space of trigonometric polynomials
\begin{equation}\label{Ch4spacetrigpol}
S_N=\text{span}\left\{e^{ikx} \ \middle| \ -N\leq k \leq N-1\right\}.
\end{equation}

Let $u(x,t)$ defined in $[0,2\pi]\times[0,T]$ be a continuous function. We define the set of nodes $\{x_j\}_{j=0}^{2N-1}$ by
\begin{equation}\label{Ch4setofnodes}
x_j=j\frac{\pi}{N}, \ j=0,1,...,2N-1,
\end{equation}

The Discrete Fourier Transform (DFT) coefficients $\left\{\hat{u}_k(t)\right\}_{k=-N}^{N-1}$ are
\begin{equation}\label{Ch4dft}
\hat{u}_k(t)=\frac{1}{2N}\sum^{2N-1}_{j=0}u(x_j,t)e^{-ikx_j}, \quad k=-N,...,N-1.
\end{equation}
and the trigonometric interpolant of function $u(x,t)$ at $\{x_j\}_{j=0}^{2N-1}$ is given by
\begin{equation}
I_N(u(x,t))=\sum^{N-1}_{k=-N}\hat{u}_k(t)e^{ikx}
\end{equation}
where the $\left\{\hat{u}_k(t)\right\}_{k=-N}^{N-1}$ are given by (\ref{Ch4dft}).

Let $u^N\in S_N$. The polynomial $u^N$ is unambiguously defined by its values at the nodes $\{x_j\}_{j=0}^{2N-1}$ given by (\ref{Ch4setofnodes}). We denote
\begin{equation}
\boldsymbol{U}_N=\left[u^N(x_0),...,u^N(x_{2N-1})\right]^{T}.
\end{equation}

The Discrete Fourier Transform (DFT) is an invertible, linear transformation $\mathfrak{F}_N:\mathbb{C}^{2N}\longrightarrow\mathbb{C}^{2N}$. We define
\begin{equation}\label{Ch4coeftransfour}
\hat{\boldsymbol{U}}_N=[\hat{u}^N_{-N},...,\hat{u}^N_0,...,\hat{u}^N_{N-1}]=\mathfrak{F}_N\boldsymbol{U}_N,
\end{equation}

The spectral derivative, \cite{Canuto}, is given by:
\begin{equation*}
D_N\boldsymbol{U}_N={\mathfrak{F}^{-1}_N\Delta_N \mathfrak{F}_N}\boldsymbol{U}_N \quad (\text{recursively} \  D^k_N={\mathfrak{F}^{-1}_N\Delta^k_N \mathfrak{F}}_N),
\end{equation*}
where $\Delta_N$ is a diagonal matrix given by $\Delta_N=\text{diag}(i n:-N\leq n \leq N-1)$.

For the rest of the work, given a complex function $u(x,t)$ defined in $[0,2\pi]\times[0,T]$, the notation $u(t)$ refers to a function $u(\cdot,t)\in L^2\left([0,2\pi],\mathbb{C}\right)$.

Let $u_T(x)$ be a given function. The Fourier collocation method, \cite{Canuto}, for equation (\ref{Ch4ecunolinealdeftercercam}) supplemented with periodic boundary conditions and subject to $u(x,T)=u_T(x)$ consists in finding a trigonometric polynomial $u^N(t)\in S_N$ such that $\forall j=0,1,...,2N-1$:
\begin{equation}\label{Ch4fourcolmeth}
\begin{aligned}
\frac{\partial u^N(x_j,t)}{\partial t}+A\frac{\partial u^N(x_j,t)}{\partial x}+B \frac{\partial^2 u^N(x_j,t)}{\partial {x}^2}+C  \left(\frac{\partial u^N(x_j,t)}{\partial {x}}\right)^2=0, \\
u^N(x_j,T)=u_T(x_j).
\end{aligned}
\end{equation}

The partial differential equation can be written as
\begin{equation*}
\frac{\partial \boldsymbol{U}_N}{\partial t}+A D_N\boldsymbol{U}_N+BD_N^2\boldsymbol{U}_N+ C \left(D_N\boldsymbol{U}_N\circ D_N\boldsymbol{U}_N\right)=0,
\end{equation*}
where $\circ$ denotes the Hadamard (entrywise) product.

Alternatively, using that $\hat{\boldsymbol{U}}_N=\mathfrak{F}_N\boldsymbol{U}_N$,
\begin{equation}
\frac{\partial \hat{\boldsymbol{U}}_N}{\partial t}+A\Delta_N\hat{\boldsymbol{U}}_N+B\Delta^2_N\hat{\boldsymbol{U}}_N+C\mathfrak{F}_N\left(\mathfrak{F}_N^{-1}\Delta_N \hat{\boldsymbol{U}}_N\circ\mathfrak{F}_N^{-1}\Delta_N \hat{\boldsymbol{U}}_N\right)=0,
\end{equation}
which is condensed as
\begin{equation}\label{Ch4expresfour}
\frac{\partial \hat{\boldsymbol{U}}_N}{\partial t}=\text{L}\left(\hat{\boldsymbol{U}}_N\right)+\text{NL}\left(\hat{\boldsymbol{U}}_N\right),
\end{equation}
where
\begin{equation}
\begin{aligned}
& \text{L}\left(\hat{\boldsymbol{U}}_N\right)=-\left[A\Delta_N+B\Delta^2_N\right]\hat{\boldsymbol{U}}_N, \\
& \text{NL}\left(\hat{\boldsymbol{U}}_N\right)=-C\mathfrak{F}_N\left(\mathfrak{F}_N^{-1}\Delta_N \hat{\boldsymbol{U}}_N\circ\mathfrak{F}_N^{-1}\Delta_N \hat{\boldsymbol{U}}_N\right).
\end{aligned}
\end{equation}

Expression (\ref{Ch4expresfour}) is equivalent to the collocation equation (\ref{Ch4fourcolmeth}). For recovering the function values at the nodes, we just apply the inverse operator $\boldsymbol{U}_N=\mathfrak{F}^{-1}_N\hat{\boldsymbol{U}}_N$ when necessary.

The numerical solution of (\ref{Ch4ecunolinealdeftercercam}) subject to $u(x,T)=u_T(x)$ is the polynomial $u^N(x,t)$ such that
\begin{equation}
\boldsymbol{U}_N(t)=\left[u^{N}(x_0,t),...,u^{N}(x_{2N-1},t)\right]^{T},
\end{equation}
which satisfies
\begin{equation}
\begin{aligned}
\frac{\partial \hat{\boldsymbol{U}}_N}{\partial t}=\text{L}\left(\hat{\boldsymbol{U}}_N\right)+\text{NL}\left(\hat{\boldsymbol{U}}_N\right), \\
\boldsymbol{U}_N(T)=\left[u_T(x_0),...,u_T(x_{2N-1})\right]^{T}.
\end{aligned}
\end{equation}

We refer to Subsection \ref{OPTCEUCSCE} for the theoretical analysis of the stability and convergence of the pseudospectral method. We now proceed to give the computational algorithm.

\subsection{Numerical algorithm}\label{OPTCEUtPmDPNV}

Suppose that we want to compute option prices for $\hat{x}$ in the \textit{approximation domain}, $\hat{x}\in[L_{\min},L_{\max}]$.

Therefore, we want to obtain a numerical solution for:
\begin{equation*}
H_j(t,y,\hat{x}):[0,T]\times[y_{\min},y_{\max}]\times[\hat{x}_{\min},\hat{x}_{\max}]\longrightarrow \mathbb{R}.
\end{equation*}
where $y_{\min}$, $y_{\max}$, $\hat{x}_{\min}$ and $\hat{x}_{\max}$ are chosen to be big enough. We refer to Section \ref{OPTCEUmda} for the empirical error analysis (localization error/number of shares).

\begin{definition}\label{Ch4defindiscretvar}
Given $\textbf{N}=(N_t, \ N_y, \ N_{\hat{x}})\in\mathbb{N}^3$, we define:
\begin{equation}
\Delta y = \frac{y_{\max}-y_{\min}}{N_y}, \quad \Delta \hat{x} = \frac{\hat{x}_{\max}-\hat{x}_{\min}}{N_{\hat{x}}}, \quad \Delta t = \frac{T}{N_t},
\end{equation}
and the sets of points
\begin{equation}
\begin{aligned}
& \{y_l\}^{N_y}_{l=0}, &&y_l=y_{\min}+l\Delta y, \\
& \{\hat{x}_k\}^{N_{\hat{x}}}_{k=0},  &&\hat{x}_k=\hat{x}_{\min}+k\Delta \hat{x}, \\
& \{t_m\}^{N_t}_{m=0},  &&t_m=m\Delta t. \\
\end{aligned}
\end{equation}
\end{definition}

For the localization procedure, we define two auxiliary sets of points.

\begin{definition}\label{Ch4defindiscretvar2}
We define $N_{x}=4N_{\hat{x}}$ and denote $\textbf{N}_e=(N_t, \ N_y, \ N_{x})\in\mathbb{N}^3$. We define:
\begin{equation}
\Delta \hat{x}^e= \frac{4\hat{x}_{\max}-3\hat{x}_{\min}}{N_{x}}=\left(\Delta \hat{x}\right), \quad \Delta x= \frac{2\pi}{N_{x}},
\end{equation}
and the sets of points
\begin{equation}
\begin{aligned}
& \{\hat{x}^e_s\}^{N_{x}}_{s=0}, && \hat{x}^e_s=\hat{x}_{\min}+s\Delta \hat{x}^e, \\
& \{x_s\}^{N_{x}}_{s=0}, && x_s=s\Delta x. \\
\end{aligned}
\end{equation}
\end{definition}

We note that $\hat{x}_k=\hat{x}^e_k, \ k=0,1,...,N_{\hat{x}}$. The set of spatial nodes $\{\hat{x}^e_s\}^{N_{x}}_{s=0}$ is needed in order to define the odd-even extension given by (\ref{Ch4extenfor1}).

The numerical solution is denoted by $H^{\textbf{N}}_j, \quad j\in\{1,w\}$. This solution is only computed for the discrete values included in $\{y_l\}^{N_y}_{l=0}$ and $\{t_m\}^{N_t}_{m=0}$.

We remark that $H^{\textbf{N}}_j, \ j\in\{1,w\}$ is the numerical approximation to the function value just in $[\hat{x}_{\min},\hat{x}_{\max}]$ but, for a particular choice of $y_{l_0}$ and $t_{m_0}$, the functions $H^{\textbf{N}}_j(t_{m_0},y_{l_0},\hat{x}), \quad j\in\{1,w\}$ are a $N_{x}=4N_{\hat{x}}$ degree trigonometric polynomial defined in $[\hat{x}_{\min},4\hat{x}_{\max}-3\hat{x}_{\min}]$ by its values at $\{\hat{x}^e_s\}^{N_{x}}_{s=0}$ after performing the odd-even extension given in (\ref{Ch4extenfor1}).

The algorithm is:

\

\noindent\textbf{{Step 0:}} Set $m=N_t$ ($t_{N_t}=T$).

For each $y_l\in\{y_l\}^{N_y}_{l=0}$ and for each ${\hat{x}_k}\in\{\hat{x}_k\}^{N_{\hat{x}}}_{k=0}$ compute
\begin{equation*}
H^{\textbf{N}}_j(T,y_l,{\hat{x}_k})=H_j(T,y_l,{\hat{x}_k}), \ j\in\{1,w\}
\end{equation*}
 with formulas (\ref{Ch4pde2camfin1})-(\ref{Ch4pde2camfin2}).

\

\noindent\textbf{{Step 1:}} For each $y_l\in\{y_l\}^{N_y}_{l=0}$, extend the function $H^{\textbf{N}}_j(t_m,y_l,\hat{x})$ defined in $[\hat{x}_{\min},\hat{x}_{\max}]$ to the trigonometric polynomials $\mathbb{H}^{\textbf{N}_e}_j(t_m,y_l,\hat{x})$ defined in $[\hat{x}_{\min}, 4\hat{x}_{\max}-3\hat{x}_{\min}]$ as in Subsection \ref{OPTCEUtPmnlpde}.

For each $y_l\in\{y_l\}^{N_y}_{l=0}$ and for $s=0,1,...,4N_{\hat{x}}-1$ define
\begin{equation*}
\mathbb{H}^{\textbf{N}_e}_j(t_m,y_l,\hat{x}^e_s)=\left\{
\begin{aligned}
& H^{\textbf{N}}_j(t_m,y_l,\hat{x}^e_s), \ \text{if} \ \ \hat{x}^e_s \in I_1, \\
& 2H^{\textbf{N}}_j(t_m,y_l,\hat{x}_{\max})-H^{\textbf{N}}_j(t_m,y_l,2\hat{x}_{\max}-\hat{x}^e_s), \ \text{if} \ \ \hat{x}^e_s \in I_2, \\
& \mathbb{H}^{\textbf{N}}_j(t_m,y_l,(4\hat{x}_{\max}-2\hat{x}_{\min})-\hat{x}^e_s), \ \text{if} \ \ \hat{x}^e_s \in I_3, \\
& \mathbb{H}^{\textbf{N}}_j(t_m,y_l,z), \ \text{if} \ \ \hat{x}^e_s \notin I,
\end{aligned}\right.
\end{equation*}
where $z=-\hat{x}^e_s+\hat{x}_{\min}+k(4\hat{x}_{\max}-4\hat{x}_{\min})$ and $k\in\mathbb{Z}$ is such that $z\in[\hat{x}_{\min},4\hat{x}_{\max}-3\hat{x}_{\min}]$. We recall that the intervals were given by (\ref{Ch4definterv}).

For each $y_l\in\{y_l\}^{N_y}_{l=0}$ and each $x_s\in\{x_s\}^{N_{x}}_{s=0}$ write
\begin{equation*}
u^{N_x}_j(t_m,y_l,x_s)=\mathbb{H}^{\textbf{N}_e}_j\left(t_m,y_l,\hat{x}_{\min}+\frac{4\hat{x}_{\max}-4\hat{x}_{\min}}{2\pi}x_s\right)
\end{equation*}
to obtain trigonometric polynomials defined for $x\in[0,2\pi]$. Set:
\begin{equation*}
\boldsymbol{U}^{y_l}_{N_x}=\left[u^{N_x}_j(t_m,y_l,x_0), u^{N_x}_j(t_m,y_l,x_1), ..., u^{N_x}_j(t_m,y_l,x_{N_x-1}) \right]^T.
\end{equation*}

Then, for each $y_l\in\{y_l\}^{N_y}_{l=0}$ compute the approximated No Transactions function value $u^{N_x}_j(t_{m-1},y_l,x)$ as the numerical solution of the Fourier pseudospectral method:
\begin{equation*}
\left\{
\begin{aligned}
& \frac{\partial \hat{\boldsymbol{U}}}{\partial t}+A\Delta\hat{\boldsymbol{U}}+B\Delta^2\hat{\boldsymbol{U}}+C\mathfrak{F}\left(\mathfrak{F}^{-1}\Delta \hat{\boldsymbol{U}}\circ\mathfrak{F}^{-1}\Delta \hat{\boldsymbol{U}}\right)=0, \\
&  \hat{\boldsymbol{U}}(t_m)=\mathfrak{F}\left(u^{N}_j(t_{m},y_l,\bar{\mathrm{x}})\right),
\end{aligned}
\right.
\end{equation*}
where $\hat{\boldsymbol{U}}=\hat{\boldsymbol{U}}^{y_l}_{N_x}$, $\Delta\hat{\boldsymbol{U}}=\Delta_{N_x}\hat{\boldsymbol{U}}^{y_l}_{N_x}$, $\mathfrak{F}=\mathfrak{F}_{N_{x}}$, constants $A,B,C$ are given by formula (\ref{Ch4ecunolinealdeftercercamcons}) and $\bar{\mathrm{x}}=\{x_k\}_{k=0}^{N_x-1}$.

For each $y_l\in\{y_l\}^{N_y}_{l=0}$, and each $\hat{x}_k\in\{\hat{x}_k\}^{N_{\hat{x}}}_{k=0}$ define
\begin{equation*}
H^p_j(t_{m-1},y_l,\hat{x}_k)=u^{N}_j(t_{m-1},y_l,x_k),
\end{equation*}
which corresponds to the function values if no transactions are realized. We remark that $\hat{x}_k\in[\hat{x}_{\min},\hat{x}_{\max}]$, the values that correspond to the computational domain.

\

\noindent\textbf{Step 2:} Search the location of the buying/selling frontiers for each $\hat{x}_k\in \{\hat{x}_k\}^{N_{\hat{x}}}_{k=0}$ at $t=t_{m-1}$.

We assume that the state space remains divided in three Regions (Buying/Selling/No Transactions).

The location of the frontiers is done through the discrete counterpart of equation (\ref{Ch4ecupdedefin}) after the changes (\ref{Ch4logscale})-(\ref{Ch4logfunc}). We search the biggest/smallest value for which it is not optimal to respectively buy/sell shares. The numerical approximation to the Buying frontier is $y^{\mathscr{B}}_j(t_{m-1}, \hat{x}_k)=$ \small
\begin{equation*}
\underset{y\in\{y_l\}^{N_y}_{l=0}}{\min}\left\{\frac{\gamma(1+\lambda)\exp({\hat{x}_k})}{\delta(T,t_{m-1})}\Delta y+H^p_j(t_{m-1},y_l+\Delta y,\hat{x}_k)-H^p_j(t_{m-1},y_l,\hat{x}_k)>0\right\},
\end{equation*}\normalsize
and to the Selling frontier is $y^{\mathscr{S}}_j(t_{m+1}, \hat{x}_k)=$ \small
\begin{equation*}
\underset{y\in\{y_l\}^{N_y}_{l=0}}{\max}\left\{-\frac{\gamma(1-\mu)\exp({\hat{x}_k})}{\delta(T,t_{m-1})}\Delta y+H^p_j(t_{m-1},y_l-\Delta y,\hat{x}_k)-H^p_j(t_{m-1},y_l,\hat{x}_k)>0\right\}.
\end{equation*} \normalsize

With this definition, the discrete frontier is a point of the mesh $\{y_l\}^{N_y}_{l=0}$, so that the time evolution is piecewise constant.

\

\noindent\textbf{Step 3:} Obtain, for each $\hat{x}_k\in\{\hat{x}_k\}^{N_{\hat{x}}}_{k=0}$ and each $y_l\in\{y_l\}^{N_y}_{l=0}$ the value of $H^{\textbf{N}}_j(t_{m-1},y_l,\hat{x}_k)$ employing the explicit formulas (\ref{Ch4ecuzonacompranolin}) and (\ref{Ch4ecuzonaventanolin}).

In the Buying Region $\left(y_{l}< y^{\mathscr{B}_{\textbf{N}}}_j(t_{m-1},\hat{x}_k)\right)$, function $H^{\textbf{N}}_j(t_{m-1},y_l,\hat{x}_k)=$
\begin{equation*}
H^p_j(t_{m-1},y^{\mathscr{B}_{\textbf{N}}}_j(t_{m-1},\hat{x}_k),\hat{x}_k)+\left(-\frac{\gamma(1+\lambda)\exp({\hat{x}_k})}{\delta(T,t_{m-1})}(y_l-y^{\mathscr{B}_{\textbf{N}}}_j(t_{m-1},\hat{x}_k,))\right),
\end{equation*}

In the No Transactions $\left(y^{\mathscr{B}_{\textbf{N}}}_j(t_{m-1},\hat{x}_k) \leq y_l\leq y^{\mathscr{S}_{\textbf{N}}}_j(t_{m-1},\hat{x}_k)\right)$
\begin{equation*}
H^{\textbf{N}}_j(t_{m-1},y_l,\hat{x}_k) =H^p_j(t_{m-1},y_l,\hat{x}_k),
\end{equation*}

In the Selling Region $\left(y_l> y^{\mathscr{S}_{\textbf{N}}}_j(t_{m-1},\hat{x}_k)\right)$, function $H^{\textbf{N}}_j(t_{m-1},y,\hat{x}_k)=$
\begin{equation*}
H^p_j(t_{m-1},y^{\mathscr{S}_{\textbf{N}}}_j(t_{m-1},\hat{x}_k,),\hat{x}_k)+\left(\frac{\gamma(1-\mu)\exp({\hat{x}_k})}{\delta(T,t_{m-1})}(y^{\mathscr{S}_{\textbf{N}}}_j(t_{m-1},\hat{x}_k,)-y)\right).
\end{equation*}

\noindent\textbf{Step 4:} If $t_{m-1}=0$ end. Otherwise, $m=m-1$ and proceed to \textbf{Step 1}.

For each $y_l\in\{y_l\}^{N_y}_{l=0}$ and each $t_m\in\{t_m\}^{N_t}_{m=0}$, redefine $H^{\textbf{N}}_j(t_{m},y_l,\hat{x})$ as the trigonometric
polynomial defined $[\hat{x}_{\min},4\hat{x}_{\max}-3\hat{x}_{\min}]$ by its values at $\hat{x}^e_s\in\{\hat{x}^e_s\}^{N_{x}}_{s=0}$ with the odd-even extension given by (\ref{Ch4extenfor1}).

\

The numerical approximation to the option price $\forall \hat{x}\in[\hat{x}_{\min},\hat{x}_{\max}]$ and for each $ t_m\in\{t_m\}^{N_t}_{m=0}$  is computed through
\begin{equation}\label{Ch4numericalopprice}
p^{\textbf{N}}_w(t_m,\hat{x})=\frac{\delta(T,t_m)}{\gamma}\left(H^{\textbf{N}}_w(t_m,0,\hat{x})-H^{\textbf{N}}_1(t_m,0,\hat{x})\right).
\end{equation}

\subsection{Stability, consistency and convergence. Localization error.}\label{OPTCEUCSCE}

We will follow the lines presented in \cite{Frutos2} to study the stability and convergence of the Fourier pseudospectral method. Since partial differential equation (\ref{Ch4ecunolinealdeftercercam}) is solved backwards, for simplicity, we perform the change of variable $\tau=T-t$, so that we deal with the non-linear periodic problem:
\begin{equation}\label{ecu1}
\begin{aligned}
& \frac{\partial u}{\partial \tau}=A\frac{\partial^2 u}{\partial {x}^2}+B\frac{\partial u}{\partial x}+C\left(\frac{\partial u}{\partial {x}}\right)^2, \\
& u(0,\tau)=u(2\pi,\tau), \quad u_x(0,\tau)=u_x(2\pi,\tau), \\
& u(x,0)=u_0(x)
\end{aligned}
\end{equation}
where $u_0(x)$ is given and constants $A,B,C$ are the same as in (\ref{Ch4ecunolinealdeftercercamcons}).

For the analysis, we assume that the regularity conditions upon $u_0(x)$ are the same regularity conditions required upon $u$ in the different Theorems and Propositions. The particular initial conditions of the financial problem that we are dealing with will be discussed after the theoretical development.

Let $L^2=L^2(0,2\pi)$ denote the space of the Lebesgue-measurable functions $u:(0,2\pi)\rightarrow\mathbb{C}$. We denote by $\|.\|$ the usual $L^2$ norm \cite[(2.1.11)]{Canuto}.

We define the norm $||u||_{\infty}$ (see \cite[5.1.3]{Canuto})  by $||u||_{\infty}=\underset{0\leq x \leq 2\pi}{\sup}\left|u(x)\right|$.

For any function $u(\tau)\in L^2([0,2\pi])$, let $P_Nu(\tau) \in S_N$ be the orthogonal projection \cite[(2.1.8)]{Canuto} of $u(\tau)$ over $S_N$.

For $u\in S_N$, we denote by $||u||_N$ the usual discrete norm \cite[(2.1.34)]{Canuto}. We note that if $u\in S_{N}$, it holds $\|u\|_N=\|u\|$ (see \cite[(2.1.33)]{Canuto}).

Let $H^s=H^s(0,2\pi)$ denote the usual Sobolev space of order $s$ and $||.||_{H^s}$ its norm \cite[A.11]{Canuto}. We consider \cite{Canuto} the subspace $H^s_p\subset H^s$ defined by $H^{s}_p(0,2\pi)=$
\begin{equation*}
\begin{aligned}
\left\{v\in L^2(0,2\pi): \right. & \left. \text{for} \ 0\leq k\leq s, \ \text{the derivative} \ \frac{\text{d}^k v}{\text{d}x^k} \  \text{in the} \right. \\
& \left. \text{sense of periodic distributions belongs to} \ L^2(0,2\pi)\right\}.
\end{aligned}
\end{equation*}

\subsubsection{Stability, consistency and convergence.}

We recall that in the proposed collocation method, we search for a function $u^N(\tau)\in S_N$ such that $\forall j=0,... \ ,2N-1$:
\begin{equation}\label{Ch4colloc2}
\begin{aligned}
& \frac{\partial u^{N}}{\partial \tau}{(x_j,\tau)}=A\frac{\partial^2 u^{N}}{\partial {x}^2}{(x_j,\tau)}+B\frac{\partial u^{N}}{\partial x}{(x_j,\tau)}+C\left(\frac{\partial u^{N}}{\partial {x}}{(x_j,\tau)}\right)^2, \\
& u^{N}(0,\tau)=u^{N}(2\pi,\tau), \\
& u^{N}(x_j,0)=u_0(x_j).
\end{aligned}
\end{equation}

Fix T>0. Let $V(x,\tau), W(x,\tau)$ be two $2\pi$-periodic and smooth functions defined in $[0, \ 2\pi]\times[0, \ T]$. These functions will be seen as perturbed solutions of equation (\ref{ecu1}).

Let $V^N(\tau)=I_N(V(\tau))$ and $W^N(\tau)=I_N(W(\tau))$. We define the residuals $F^N(x,\tau)$, $G^N(x,\tau)\in S_N$, as the trigonometric polynomials such that for $j=0,... \ ,2N-1$  satisfy:
\small
\begin{equation*}
\begin{aligned}
F^N(x_j,\tau) &= \frac{\partial V^N}{\partial \tau}{(x_j,\tau)} -A\frac{\partial^2 V^N}{\partial {x}^2}{(x_j,\tau)}-B\frac{\partial V^N}{\partial x}{(x_j,\tau)}-C\left(\frac{\partial V^N}{\partial {x}}{(x_j,\tau)}\right)^2, \\
G^N(x_j,\tau) &= \frac{\partial W^N}{\partial \tau}{(x_j,t)} -A\frac{\partial^2 W^N}{\partial {x}^2}{(x_j,\tau)}-B\frac{\partial W^N}{\partial x}{(x_j,\tau)}-C\left(\frac{\partial W^N}{\partial {x}}{(x_j,\tau)}\right)^2.
\end{aligned}
\end{equation*}
\normalsize

The proofs of the following results can be found in the appendix.

\begin{theorem}\label{estabilidad}(Stability)
Let T>0 be fixed and $V^N, \ W^N, \ F^N, \ G^N$ defined above.

Let $M  \geq 0$ such that threshold condition (justified in Proposition \ref{threshold}) holds:
\begin{equation}\label{Ch4threshcondstab}
\|(V^N)_x\|_{\infty}, \|(W^N)_x\|_{\infty} \leq M, \ \ \tau\in[0,T].
\end{equation}

Then, it exists a constant $R=R\left(M\right)$ such that
\begin{equation*}
\underset{0\leq \tau \leq T}{\max} \|e^{N}(\tau)\|^2 +\frac{A}{2}\int_0^T\|e^{N}_x(\tau)\|^2d\tau  \leq R\left(\|e^{N}(0)\|^2+ \int_0^T \|J^{N}(\tau)\|^2 d\tau \right),
\end{equation*}
where $e^{N}(\tau)=V^{N}(\tau)-W^{N}(\tau)$ and $J^{N}(\tau)=F^N(\tau)-G^N(\tau)$.
\end{theorem}

\begin{proposition}\label{Ch4lemmainterpderiv}
Let $u(\tau)\in H^{s+r}_p, \ s,r\geq 1$ for $\tau\in[0,T]$ continuous.

Then it exists a constant $M=M\left(\underset{0\leq\tau\leq T}{\max}\left\|\frac{\partial^{s+1} u}{\partial x^{s+1}}\right\|\right)\geq0$ such that for any $N\in\mathbb{N}$, it holds:
\begin{equation*}
\left\|\frac{\partial^s P_N(u(\tau))}{\partial x^s}\right\|_{\infty} \leq M, \quad \left\|\frac{\partial^s I_N(u)}{\partial x^s}\right\|_{\infty} \leq M, \quad \tau\in[0,T].
\end{equation*}
\end{proposition}

\begin{proposition}\label{Ch4Consistproposition}(Consistency)
Let $u(x,\tau)$ be the solution of equation (\ref{ecu1}). Suppose that $\forall \tau\in [0,T]$, function $u(\tau)\in H^{s+2}_p$ and $u_{\tau}(\tau)\in H^s_p$.

Define $F^N(\tau)\in S_N, \ \forall \tau\in[0,T]$ and $j=0,... \ ,2N-1$ by
\begin{equation}\label{Ch4exprthconsistencia}
F^N(x_j,\tau) = \left.\left[\frac{\partial I_N(u)}{\partial \tau} -A\frac{\partial^2 I_N(u)}{\partial {x}^2}-B\frac{\partial I_N(u)}{\partial x}-C\left(\frac{\partial I_N(u)}{\partial {x}}\right)^2\right]\right|_{{(x_j,\tau)}}.
\end{equation}

Then it exists a constant
\begin{equation*}
M=M\left(\underset{0\leq \tau\leq T}{\max}\left\{\left\|\frac{\partial^{s+1} u}{\partial x^{s+1}}\right\| , \left\|\frac{\partial^{s+2} u}{\partial x^{s+2}}\right\| , \ \left\|\frac{\partial^{s} u_{\tau}}{\partial x^{s}}\right\| \right\}\right),
\end{equation*}
such that
\begin{equation*}
\underset{0\leq \tau\leq T}{\max}\|F^N\|\leq MN^{-s}.
\end{equation*}

\end{proposition}

The following result ensures that threshold condition (\ref{Ch4threshcondstab}) holds for $u^N$.

\begin{proposition}\label{threshold}
Fix T>0. Let $u$ be the solution of equation (\ref{ecu1}). Suppose that $\forall \tau\in[0,T]$, functions $u(\tau)$  and $u_{\tau}(\tau)$ are in $H^{s+2}_p$ and $H^{s}_p$ respectively.

Then, it exists a constant $M$ and $N_0\in \mathbb{N}$, such that $\forall N\geq N_0$ it holds:
\begin{equation}\label{Ch4ecucondthres}
\|(u^N)_x(\tau)\|_{\infty} \leq M, \ \ \tau\in[0,T].
\end{equation}
\end{proposition}

\begin{theorem}\label{convergencia}(Convergence)

Let $u(\tau)$ be the solution of (\ref{ecu1}). Suppose that $u(\tau)$, $u_{\tau}(\tau)$ are respectively functions in $H^{s+2}_p$ and $H^{s}_p$ and continuous with respect $\tau\in[0,T]$.

Then, if $u^{N}(x,\tau)$ is the approximation obtained by the collocation method (\ref{Ch4colloc2}), it exists a constant
\begin{equation*}
M=M\left(\underset{0\leq \tau\leq T}{\max}\left\{ \left\|\frac{\partial^{s+1} u}{\partial x^{s+1}}\right\|, \left\|\frac{\partial^{s+2} u}{\partial x^{s+2}}\right\| , \ \left\|\frac{\partial^{s} u_{\tau}}{\partial x^{s}}\right\| \right\}\right),
\end{equation*}
and $N_0\in\mathbb{N}$ such that $\forall N\geq N_0$ it holds
\begin{equation*}
\underset {0\leq \tau \leq T}{\max} \left\{\|u(\tau)-u^{N}(\tau)\|\right\} \leq MN^{-s}.
\end{equation*}

\end{theorem}

\subsubsection{Comments about threshold condition in our financial problem.}

In the previous Subsection we have given general regularity conditions that guarantee the results of stability, consistency and convergence. We study now the regularity of the initial condition.

Note that $u(0)=u_0$ is explicitly given. This is relevant in Proposition \ref{threshold} (Threshold condition)
\begin{equation*}
\|(u^N)_x(0)\|_{\infty}=\left\|(I_N(u_0))_x\right\|_{\infty}\leq M_1
\end{equation*}
where $M_1$ is independent of $N$. In Theorem \ref{convergencia} (Convergence), we have to check:
\begin{equation*}
\underset {0\leq \tau \leq T}{\max} \left\{\|u(\tau)-u^{N}(\tau)\|\right\} \leq MN^{-s}.
\end{equation*}
which implies that we have to study $\|u(0)-u^{N}(0)\|=\|u_0-I_N(u_0)\|$.

In our problem, we invoke the pseudospectral method in different time steps (see Subsection \ref{OPTCEUtPmDPNV}). We solve equation (\ref{ecu1}) with different initial conditions which correspond to a certain function
\begin{equation*}
u_0=u_j(t_m,y_k,x), \ j\in\{1,w\},
\end{equation*}
where $t_m$ and $y_m$ are values from the time and number of shares meshes respectively and functions $u_j(t,y,x), \ j\in\{1,w\}$ were defined in (\ref{Ch4ecunolinealdeftercam}).

Functions $u_j(t,y,x), \ j\in\{1,w\}$ were constructed from $H_j(t,y,\hat{x}), j\in\{1,w\}$ after performing the odd-even extension, imposing periodic boundary conditions and a change of variable to $[0,2\pi]$.

For $t_m=T$, function $H_w(T,y_k,x)$ is continuous but not differentiable and, in general, the odd-even extension procedure does not  give differentiable functions, even when applied to differentiable functions.

\

\noindent\textbf{1. Cases $u_0=u_j(t_m,y_k,x), \ j\in\{1,w\}, \ t_m\neq T$ and $u_0=u_1(T,y_k,x)$:}

For $t_m\in[0,T), \ j\in\{1,w\}$, the conditions
\begin{equation}
\|u_0-I_N(u_0)\|\leq MN^{-2},  \quad \left\|(I_N(u_0))_x\right\|_{\infty}\leq M_1
\end{equation}
have a justification based in the following result.

\begin{proposition} \label{penulprop}
Let $f(x), \ x\in\left(0,\frac{\pi}{2}\right)$ be a twice derivable function such that $f'(0^{+})=0$ and $f'\left(\frac{\pi^{-}}{2}\right)$, $f''\left(0^{+}\right)$, $f''\left(\frac{\pi^{-}}{2}\right)$ exist.

Let $f^e(x)$ be the function which corresponds to the odd-even extension given by (\ref{Ch4extenfor1}).

It holds that:
\begin{equation*}
\|f^e-I_N(f^e)\|\leq K_1N^{-s}, \ s=2.
\end{equation*}
\end{proposition}

Up to the change of variable
\begin{equation*}
x=\hat{x}_{\min}+\frac{4\hat{x}_{\max}-4\hat{x}_{\min}}{2\pi}x,
\end{equation*}
note that for $t_m\neq T$, function $H_j(t_m,y_k,\hat{x}), \ j\in\{1,w\}$ defined in $\hat{x}\in[\hat{x}_{\min},\hat{x}_{\max}]$ plays the role of $f(x)$ and $u_j(t_m,y_k,x)=H^e_j(t_m,y_k,\hat{x})$ plays the role of $f^e(x)$ of the previous Proposition.

The result has to be applied in the limit $\hat{x}_{\min}\rightarrow -\infty$ and for $H_j$ regular enough. For the case when there are no transaction costs, the regularity and that
\begin{equation*}
\underset{\hat{x}\rightarrow -\infty}{\lim} \frac{\partial H_j}{\partial \hat{x}}(t_m,y_k,\hat{x})=0,
\end{equation*}
can be explicitly checked. We conjecture that the conditions hold when transaction costs appear.

For function $u_0=u_1(T,y_k,x)$ the same argument can be applied.

\

\noindent\textbf{2. Case $u_0=u_w(T,y_k,x)$:}

This initial condition has to be studied independently.

\begin{proposition}\label{ultimprop}

For $u_0=u_w(T,y_k,x)$, it holds that
\begin{equation*}
\|(I_N(u_0))_{x}\|_{\infty} \leq KN^{\frac{3}{2}}\|I_N(u_0)-u_0\|+C.
\end{equation*}
\end{proposition}

Therefore, the only thing that remains to check is the behaviour of $\|I_N(u_0)-u_0\|$. We empirically study the $L^2$  interpolation error.  We compute, for $N=\{128,256,512,1024,2048\}$,
\begin{equation*}
\|u_1(T,y_k,x)-I_N(u_1(T,y_k,x))\|, \quad \|u_w(T,y_k,x)-I_N(u_w(T,y_k,x))\|
\end{equation*}
with the Matlab routine \textit{quad}. The empirical orders of convergence of the error are -2.95 for $u_1$ and $-1.85$ for $u_w$. This implies that the regularity condition for $u_0$ in Proposition \ref{threshold} is fulfilled for $j\in\{1,w\}$ and suggest that the expected convergence rate of the numerical solution $H^{\textbf{N}}_j, j\in\{1,w\}$ of our problem is
\begin{equation*}
\|H_w-H_w^{N}\|\leq CN^{-2}, \quad  \|H_1-H_1^{N}\|\leq CN^{-s}, \ s\geq 2
\end{equation*}
in the spatial variable.

\subsubsection{Localization error}

When we extend the function twice and we impose periodic boundary conditions, we are modifying the real terminal conditions of the partial differential equation associated with the No Transaction region and we are inducing a numerical error, called the \textit{localization error}.

If the spatial variable is not bounded, a way of studying the effect of the \textit{localization error} in a fixed domain $D$ (approximation domain) is given in \cite{Breton}. The procedure would be to check that the difference of the exact solution of the periodic problem and the exact solution of the real problem on $D$ converges to 0 as we increase the limits of the spatial variable before proceeding to the periodic extension.

\begin{remark}
\normalfont{
Note that the convergence and the localization error analysis are totally independent. In the convergence analysis we have proved that the numerical solution converges to the exact solution of the periodic problem.

In the localization error analysis we will prove that the exact solution of the periodic problem converges to the exact solution of the original problem on the approximation domain for increasing size of the computational domain.

The analysis will be performed over the partial differential equation
\begin{equation}\label{Ch4ecunolinealprimercam}
\frac{\partial Q_j}{\partial t}+\left(\alpha-\frac{\sigma^2}{2}\right) \frac{\partial Q_j}{\partial \hat{x}}+\frac{1}{2}\sigma^2 \frac{\partial^2 Q_j}{\partial \hat{x}^2}=0, \quad j\in\{1,w\},
\end{equation}
which corresponds to equation (\ref{Ch4ecunolinealdef}) where $H_j(t,y,\hat{x})=\log \left(Q_j(t,y,\hat{x})\right), \quad j\in\{1,w\}$ or, equivalently, to equation (\ref{Ch4ecdeffNT}) after the change $\hat{x}=\log(S)$.}
\end{remark}

We recall the bankruptcy function introduced in Subsection \ref{Ch4intrbankstate}. For a fixed $E=E_0$, $X=X_0$, we are going to work with functions $Q^{B_{E_0,X_0}}_j, \ \quad j\in\{1,w\}$.

We also recall Proposition \ref{Ch4Qbounded} which stated that it exists $M=M(X_0,E_0)>0$ such that
\begin{equation*}
0\leq Q^{B_{E_0,X_0}}_j \leq M, \quad j\in\{1,w\}.
\end{equation*}

Equation (\ref{Ch4ecunolinealprimercam}) has to be solved for each value of $y$, so let $y=y_0$ and $t=t_0\in[0,T]$. We define $\phi(\hat{x}) = Q^{B_{E_0,X_0}}_j(t_0,y_0,\hat{x}), \quad \text{where} \ j=1 \ \text{or} \ j=w $.

\begin{definition}\label{Ch4definitionlocalerr}
For a fixed $L>0$, we define the approximation domain $[-L,L]$. Let ${\hat{x}^{*}}>0$ be such that $[-L, L]\subset[-\hat{x}^{*},\hat{x}^{*}]$. We define the function
\begin{equation*}
\phi^{\hat{x}^{*}}_p(\hat{x})=
\left\{
\begin{aligned}
& \phi(x) && \text{if} \ \hat{x}\in[-\hat{x}^{*}, \hat{x}^{*}], \\
& 2\phi^{\hat{x}^{*}}_p(\hat{x}_0)-\phi^{\hat{x}^{*}}_p(2{\hat{x}^{*}}-\hat{x}) && \text{if} \ \hat{x}\in[\hat{x}^{*}, 3\hat{x}^{*}],  \\
& \phi^{\hat{x}^{*}}_p(6\hat{x}^{*}-\hat{x}) && \text{if} \ \hat{x}\in[3\hat{x}^{*}, 7\hat{x}^{*}],  \\
& \phi^{\hat{x}^{*}}_p(\hat{x} \ mod \left([-\hat{x}^{*}, 7 \hat{x}^{*}]\right))  && \text{if} \ \hat{x}\notin[3\hat{x}^{*}, 7\hat{x}^{*}] . \end{aligned}
\right.
\end{equation*}
\end{definition}

\begin{theorem} \label{ultimtheor} Let $R^{\hat{x}^{*}}_p(\hat{x},t)$ and $R(\hat{x},t)$ be the solutions of
\begin{equation*}
\frac{\partial Q}{\partial t}+\left(\alpha-\frac{\sigma^2}{2}\right) \frac{\partial Q}{\partial x}+\frac{1}{2}\sigma^2 \frac{\partial^2 Q}{\partial x^2}=0,
\end{equation*}
subject to $R^{\hat{x}^{*}}_p(\hat{x},t_0)=\phi^{\hat{x}^{*}}_p(\hat{x})$ and $R(\hat{x},t_0)=\phi(\hat{x})$.

Let $L>0$ and $t\leq t_0$. Then, for any $\epsilon>0$ it exists $\hat{x}_{\epsilon}>0$ such that $\forall \hat{x}^{*}\geq \hat{x}_{\epsilon}$ it holds that
\begin{equation*}
\left|R^{\hat{x}^{*}}_p(\hat{x},t)-R(\hat{x},t)\right| \leq \epsilon, \quad  \hat{x}\in[-L,L].
\end{equation*}
\end{theorem}

A numerical example of this result is presented in Subsection \ref{OPTCEUeclecc}.

\section{Numerical results}\label{OPTCEUmda}

We first note that, when no transaction costs are present ($\lambda=\mu=0$), the problem is explicitly solvable (see \cite{Davis2} and \cite{Karatzas2}). The objective functions and the optimal trading strategies are explicitly computable and $p_w(t,S)$ is indeed the Black-Scholes price of the option.

Several temporal implementations for the Fourier method have been tested: explicit Euler, the implicit midpoint rule with Newton method to solve the nonlinear equation and the linearly implicit midpoint rule.

We have chosen the last one because it gave the best results when we compared the error convergence and computational cost. This implementation is given by:

\begin{equation*}
\frac{\hat{U}^{n+1}-\hat{U}^{n}}{\Delta t}=\text{L}\left(\frac{\hat{U}^{n+1}+\hat{U}^n}{2}\right)+\text{NL}\left(\frac{3}{2}\hat{U}^{n}-\frac{1}{2}\hat{U}^{n-1}\right)
\end{equation*}

Prior to the analysis of the error convergence, we make some remarks.  When there are no transaction costs ($\lambda=\mu=0)$, we can explicitly check \cite{Davis2}, that as $S\rightarrow 0$ it holds that $y^{\mathscr{S}}_j\rightarrow \infty$.

In our numerical method, we need to employ quite small values for the logarithmic stock price. Therefore, up to a certain level, the numerical approximation of the Buying/Selling frontiers may reach the limit of the computational domain of the number of shares and we will have to truncate.

A numerical error is generated in steps 2-4 of the algorithm of Section \ref{OPTCEUtPmDPNV}, where we have to find the optimal trading strategy and recompute $H^{\textbf{N}}_j, \ j\in\{1,w\}$ in the Buying/Selling regions. This error affects the left side of the stock price domain, where the smallest stock values are. This error can be controlled (or even removed) just increasing the domain of the number of shares, something that progressively moves it more to the left of the domain of the stock until it disappears. Numerical experiments show that for $y_{\max}$ covering all the values which correspond to the approximation domain, this error has no perceptible effects in the option price.

We also mention that for $t=T$, function $H_w(T,y,x)$ is continuous but not differentiable. When approximating the function by trigonometric polynomials, this causes some oscillations, known that the Gibbs phenomena. The regularization effect of the partial differential equation smoothes out the possible singularities very fast, so the true value of function $H_w$ (and therefore the corresponding optimal trading strategies) can be rapidly approximated by its truncated Fourier series. This could be expected from the results of \cite{Canuto}. Numerical experiments suggest that the smoothing velocity depends on $\Delta t$ and $\Delta \hat{x}$.

\subsection{Error convergence, localization error and computational cost}\label{OPTCEUeclecc}

When no transaction costs are present ($\lambda=\mu=0$), we have explicit formulas (see \cite{Davis2}) to check the error behaviour of the numerical method.

For studying the error convergence, we fix an approximation domain $[L_{\min},L_{\max}]$ and a computational domain $[\hat{x}_{\min},\hat{x}_{\max}]$. We define a set of test points $\left\{\hat{x}_p\right\}_{p=0}^{N_{p}}$ (of the approximation domain):
\begin{equation*}
\hat{x}_{p}=L_{\min}+p\frac{L_{\max}-L_{\min}}{N_p}, \quad p=0,1,2,...,N_p.
\end{equation*}
and in this set of points we study the time, spatial and number of shares error convergence.

\begin{definition}
Let $f(t,y,\hat{x})$ denote the exact value of a function which can either be $H_j, \ j\in\{1,w\}$, the option price $p_w$ or the optimal trading strategies $y^{\mathscr{B}}_j, y^{\mathscr{S}}_j, \ j\in\{1,w\}$. We recall that $S=\exp(\hat{x})$.

Let $f^{\textbf{N}}$ denote the numerical approximation subject to $\textbf{N}=(N_t,N_y,N_{\hat{x}})$, which were given in Definition \ref{Ch4defindiscretvar}.

We globally define the mean square error of  the numerical approximation of function $f^{\textbf{N}}$ as
\begin{equation}
\text{RMSE}\left(f^{\textbf{N}}\right)=\sqrt{\frac{1}{N_p+1}\sum_{p=0}^{N_p}\left(f(\hat{x}_p)-f^{\textbf{N}}(\hat{x}_p)\right)^2}.
\end{equation}
\end{definition}

We fix the parameter values $\sigma$=0.1, $\alpha=0.1$, $r$=0.085, the strike and maturity ($t=0$ today)
\begin{equation*}
\left\{
\begin{aligned}
 \hat{x}_{K} &= 2 \ \ \text{(K=7.389)}, \\
 T &= 0.5 \ \ \text{(years)}.
\end{aligned}
\right.
\end{equation*}

For clarifying purposes, we point that this corresponds to an option with strike $K=7.389$ and that we compute several functions for stock prices $S_p=e^{\hat{x}_p}$ which vary from  $2.718$ to $20.085$ (i.e. for options At and (very) In/Out the money). For the number of shares, we set $y \in [0 , \ 2]$.

Unless explicitly mentioned (Subsection \ref{OPTCEUecleccLE}), we take $[L_{\min},L_{\max}]=[1,3]$ and $[\hat{x}_{\min},\hat{x}_{\max}]=[-5,5]$. The limits of the computational domain have been taken big enough in order to minimize the effect of the localization error.

\subsubsection{Spatial Error convergence}

We take $\Delta y = 2\ldotp5 \cdot 10^{-3}, \quad  \Delta t = 5 \cdot 10^{-5}$. We compute the RMSE for $N_x=\{50,100,200,400,800,1600\}$ and a fixed $N_p=10$ (all the nodes of the approximation domain for $N_{\hat{x}}=50$).

We restrict to $t=y=0$, i.e., the functions that are employed to compute the option price for a maturity of $0.5$ years.

 In Figure \ref{Ch4conespfunyestH1Hw} we represent the values in logarithmic scale of $\text{RMSE}$ of the numerical approximation to the value function $H^{\textbf{N}}_j, \ j\in\{1,w\}$ (left) and to the value of the buying/selling frontiers ($y^{\mathscr{B}}_j=y^{\mathscr{S}}_j, \ j\in\{1,w\}$) (right). We plot $j=1$ (solid-red) and $j=w$ (solid-blue).

\begin{figure}[h]
\centering
\includegraphics[width=12cm,height=5 cm]{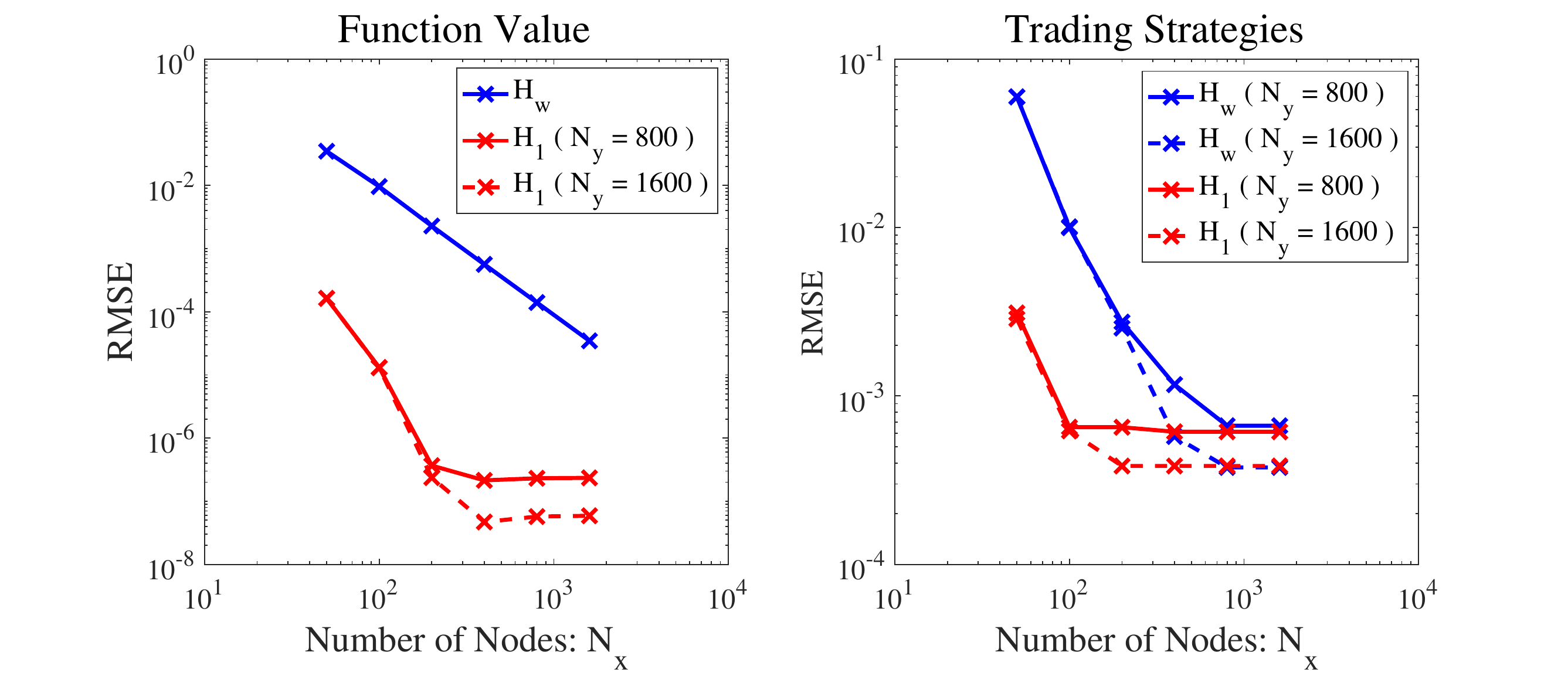}
\caption[Spatial error convergence of functions and trading strategies]{\label{Ch4conespfunyestH1Hw} Spatial error convergence of functions $H^{\textbf{N}}_1$ (left-red), $H^{\textbf{N}}_w$ (left-blue), $y^{\mathscr{B}}_1$ (right-red) and $y^{\mathscr{B}}_w$ (right-blue) in logarithmic scale.}
\end{figure}

In the left side, the slope of the regression lines (solid) are $-2\ldotp004$ for $H_w$ and $-4\ldotp397$ for $H_1$ (3 first points). We realize another experiment with $\Delta y = 1\ldotp25 \cdot 10^{-3}$ (dashed-red) to check that the lowest value of the error reached by $H_1$ (solid-red) was given by the value of $\Delta y$. $H_1$ and $H_w$ seem to present different error behaviour. This results are consistent with the upper bounds of the error convergence rate mentioned just after Theorem \ref{convergencia} (Comments about threshold condition).

In the right side, the slope of the regression line of $y^{\mathscr{B}}_w$ (solid-blue) is $-1.87$. We recall that for $t=T$, there is a jump discontinuity at $x=x_K$. We carry out a second experiment with $\Delta y = 1\ldotp25 \cdot 10^{-3}$ (dashed-blue, dashed-red) to check that the lowest value reached by the error is marked by the size of the mesh of $y$. This lowest value is reached very soon by $y^{\mathscr{B}}_1$ (solid-left).

Concerning the option price, given by (\ref{Ch4numericalopprice}), the error of function value of $H_w$ is much bigger than that of $H_1$, so the error convergence of $\text{RMSE}(p^{\textbf{N}}_w)$ is the same of function value $H_w$ (left-blue) in Figure \ref{Ch4conespfunyestH1Hw}.

The size of the error at point $\hat{x}$ depends mostly in the relative position of $\hat{x}$ with respect to $\hat{x}_K$, where the highest errors occur. For $N_x=1600$ and $S=K=7.389$ dollars, the contract value is $0.3936$ and the absolute error has been $1.09 \cdotp 10^{-4}$. We also mention that for $S\in[0.0067,4.08]$, the real option prices are  $0\sim 10^{-16}$ (dollars) and the numerical method gives $\sim 10^{-11}$. For $S>9.025$ (option prices bigger than $1.9$ dollars), the absolute errors are below $10^{-6}$.

\subsubsection{Temporal Error convergence}

In this experiment, we take the same values as the previous one for the model parameters, the strike and the computational/approximation domains. We fix $\Delta \hat{x} =6\ldotp25 \cdotp 10^{-3}$ ($N_{\hat{x}}=1600$) and $\Delta y= 2\ldotp5 \cdotp 10^{-3}$.

Value $N_t=\{1,2,4,8,16,25,50,100,200,400,800,1600,3200,10000\}$ (big values for $\Delta t$) because the size of the temporal error in this model is very small compared with other errors. For the set of test points, we fix $N_p=320$ (all the nodes of the approximation domain for $N_{\hat{x}}=1600$).

Figure \ref{Ch4contemfunyestH1Hw} shows in logarithmic scale the number of temporal nodes versus the RMSE for functions $H^{\textbf{N}}_j, j\in \{1,w\}$ (left) and the optimal trading strategies (right).

\begin{figure}[h]
\centering
\includegraphics[width=12cm,height=5 cm]{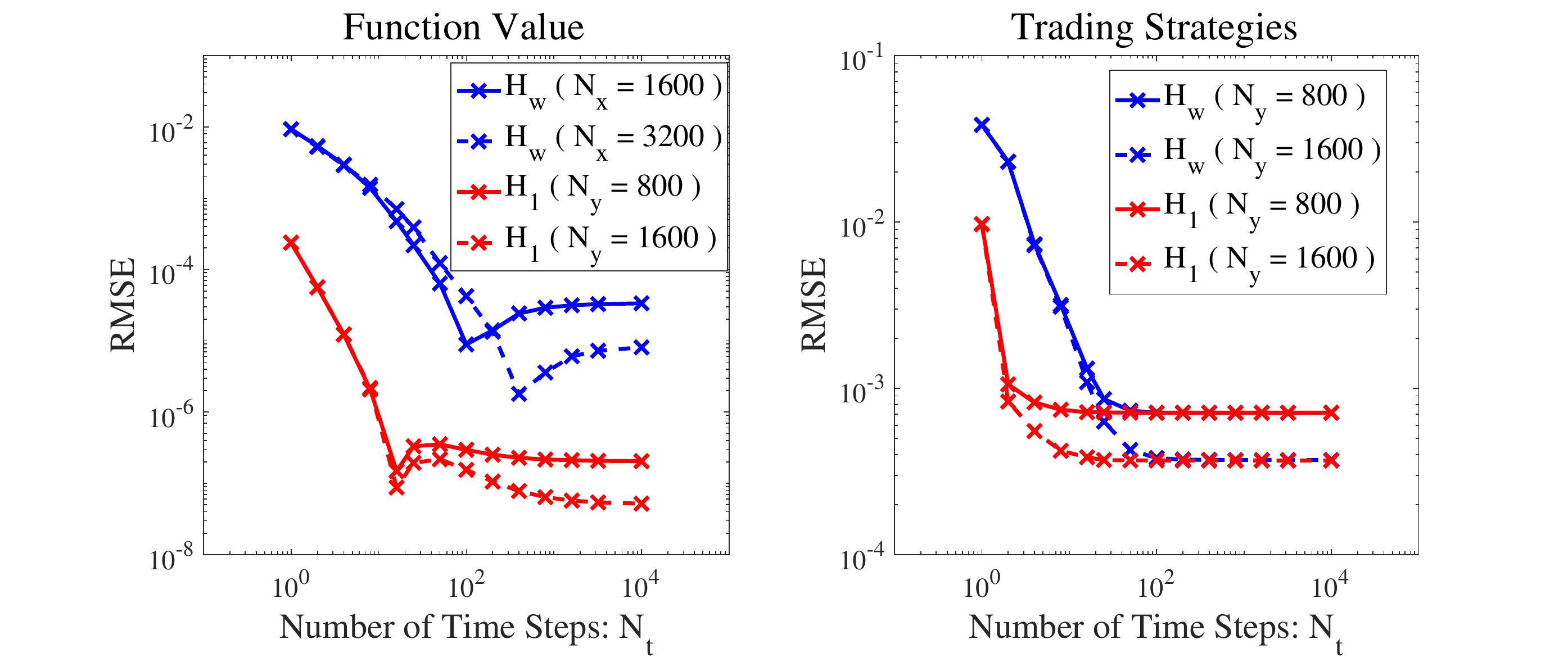}
\caption[Temporal error convergence of functions and trading strategies]{\label{Ch4contemfunyestH1Hw} Temporal error convergence of functions $H^{\textbf{N}}_1$ (left-red), $H^{\textbf{N}}_w$ (left-blue), $y^{\mathscr{B}}_1$ (right-red) and $y^{\mathscr{B}}_w$ (right-blue) in logarithmic scale.}
\end{figure}

In the left side, the slopes of the regression lines are $-2.27$ for $H_1$  (4 first points, solid-red) and $-1.26$ for $H_w$ (7 first points, solid-blue). For function $H_w$ and $N_t$ small, we may not wipe out completely the Gibbs effect and, for bigger values of $N_t$, we reach very soon the error limit marked by $\Delta \hat{x}$. Perhaps this is the reason why we do not observe an order 2 in time for function $H_1$. We carry out a second experiment halving the value of $\Delta y$ for $H_1$ (dashed-red) and the value of  $\Delta \hat{x}$ for $H_w$ (dashed-blue) to check that the lowest value reached by the errors was respectively given by the size of the meshes of the other two variables.

In the right side, the slope of the regression line of the optimal trading strategy of $H_w$ is $-1.23$ (5 first points, solid-blue). The lowest value reached by the error is given by the size of  $\Delta y$ in both cases as it can be checked in the experiment where we halve the value of $\Delta y$ (dashed-blue/red).

\subsubsection{Number of shares Error convergence}

In this experiment, we take the same values as the previous one for the model parameters and the approximation/computational domains. We fix $\Delta \hat{x} =6\ldotp25 \cdotp 10^{-3}$ ($N_{\hat{x}}=1600$), $\Delta t = 5 \cdot 10^{-5}$ and set $N_p=320$.

We are going to compute RMSE for $N_y=\{8,16,32,64,128,256,512\}$. Figure \ref{Figura5} shows the log-log of functions $H_j, \ j\in{1,w}$ (left side) and the optimal trading strategies (right side).

\begin{figure}[h]
\centering
\includegraphics[width=12cm,height=5 cm]{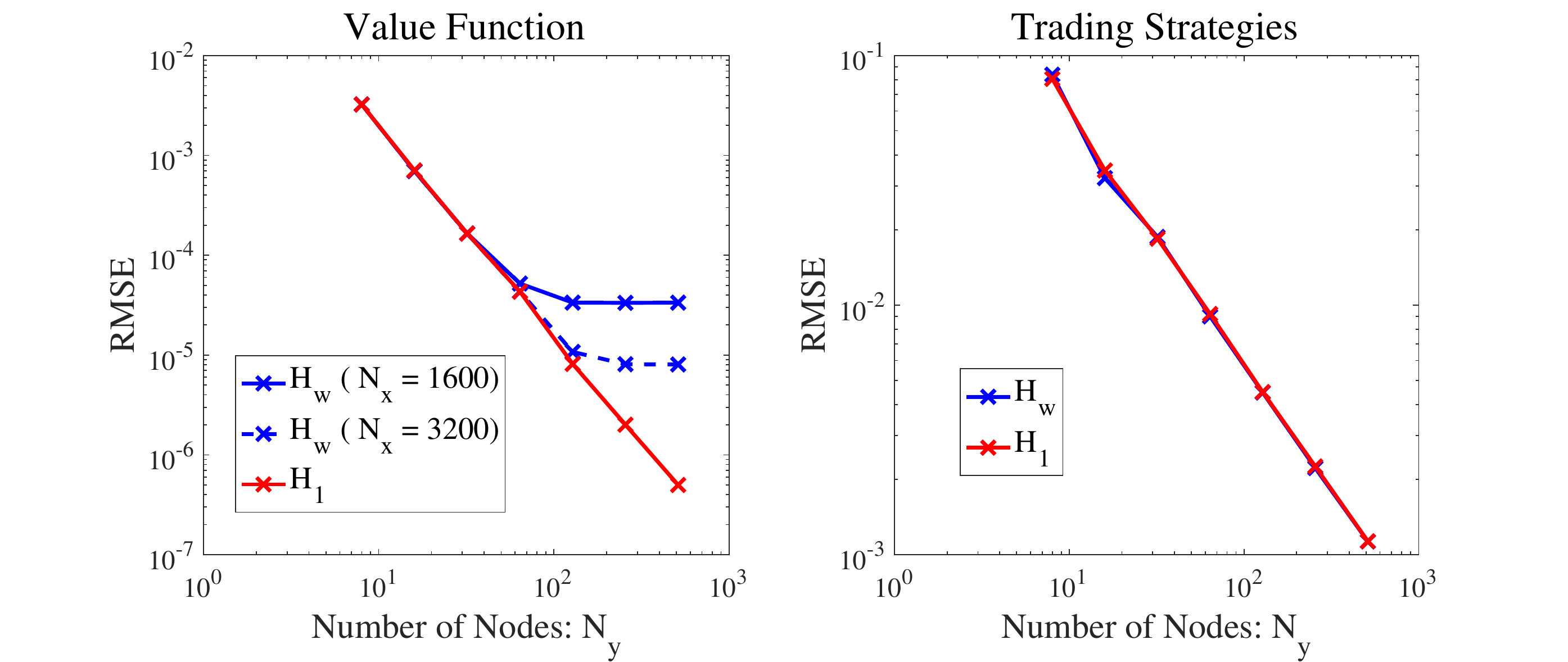}
\caption[Number of shares error convergence]{\label{Figura5} Number of shares error convergence of the value functions (left) and the optimal trading strategies (right) with $j=1$ (red) and $j=w$ (blue) in logarithmic scale.}
\end{figure}

The slope of the regression lines is $-2.11$ in the case of functions $H_j$ (red($j=1$), solid-blue($j=w$)) and $-1.01$ in the case of the optimal trading strategies (right). We carry out another experiment (dashed-blue) where $N_{\hat{x}}=3200$. This experiments shows that the lowest value of the error reached by function $H_w$ (solid-blue) was given by the size of $\Delta \hat{x}$.

Empirically, the behaviour of the computational cost has been checked to be linear in the number of time steps ($N_t$) and in the number of shares ($N_y$) and almost linear in the number of spatial nodes (theoretically $\mathscr{O}(N_{\hat{x}}\log(N_{\hat{x}}))$).

\subsubsection{Localization Error}\label{OPTCEUecleccLE}

We fix the same model parameters of the previous analysis and the same approximation domain $[L_{\min},L_{\max}]=[1,3]$. For studying the convergence of the localization error, we propose the following experiment.

The computational domain is defined by $[\hat{x}_{min},\hat{x}_{\max}]=[L_{\min}-M,L_{\max}+M]$ for $M>0$ and we define the proportion
\begin{equation*}
P(M)=\frac{L_{\max}+M-(L_{\min}-M)}{L_{\max}-L_{\min}}.
\end{equation*}

With the same values for $\Delta \hat{x}, \ \Delta y, \ \Delta t$, we compute the $\text{RMSE}$ for different values of $M$. Figure \ref{Figura6} shows the logarithm of $P(M)$ versus the logarithm of $\text{RMSE}$ for the function values and the optimal trading strategies. As it can be checked, as $M$ grows, the size of the localization error decreases as it could be expected from results of Subsection \ref{OPTCEUCSCE}. The limit error is marked by the size of $\Delta \hat{x}, \ \Delta y$ and $\Delta t$.

\begin{figure}[h]
\centering
\includegraphics[width=12cm,height=5 cm]{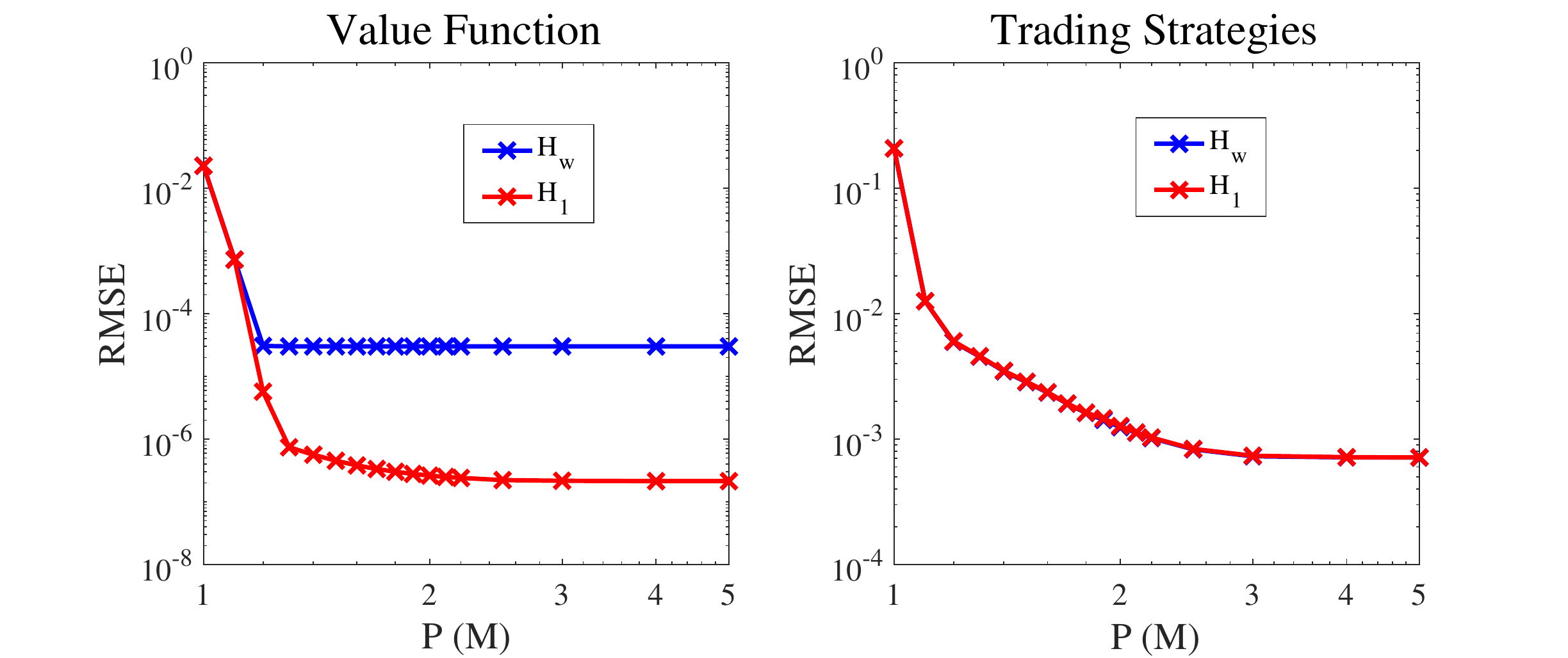}
\caption[Localization Error convergence]{\label{Figura6} Localization Error convergence. Value of $\text{RMSE}$ of the function values (left) and optimal trading strategies (right) for $j=1$ (red) and $j=w$ (blue) for different values of $P(M)$ in logarithmic scale.}
\end{figure}

To finish this Subsection, we recall that we mentioned that there was a scaling problem if we worked with the original variables. This problem has been greatly reduced with our numerical method. With the new variables, we can work with very big values for the stock and the strike ($S,K>10^4$) and also compute options very deep in the money.

\subsection{Numerical examples with transaction costs}\label{OPTCEUmdaItc}

We check now the effects of incorporating transaction costs to the pricing model. We repeat the experiments realized in \cite[Fig. 1]{Davis2}. Figure \ref{Figura7} shows the price difference for all maturities between $T\in[0,3]$, i.e. $\text{Price difference}=p_w-\text{BS},$ where $p_w$ denotes the option price with transaction costs and BS the Black-Scholes price.

\begin{figure}[h]
\centering
\includegraphics[width=7cm,height=4.5cm]{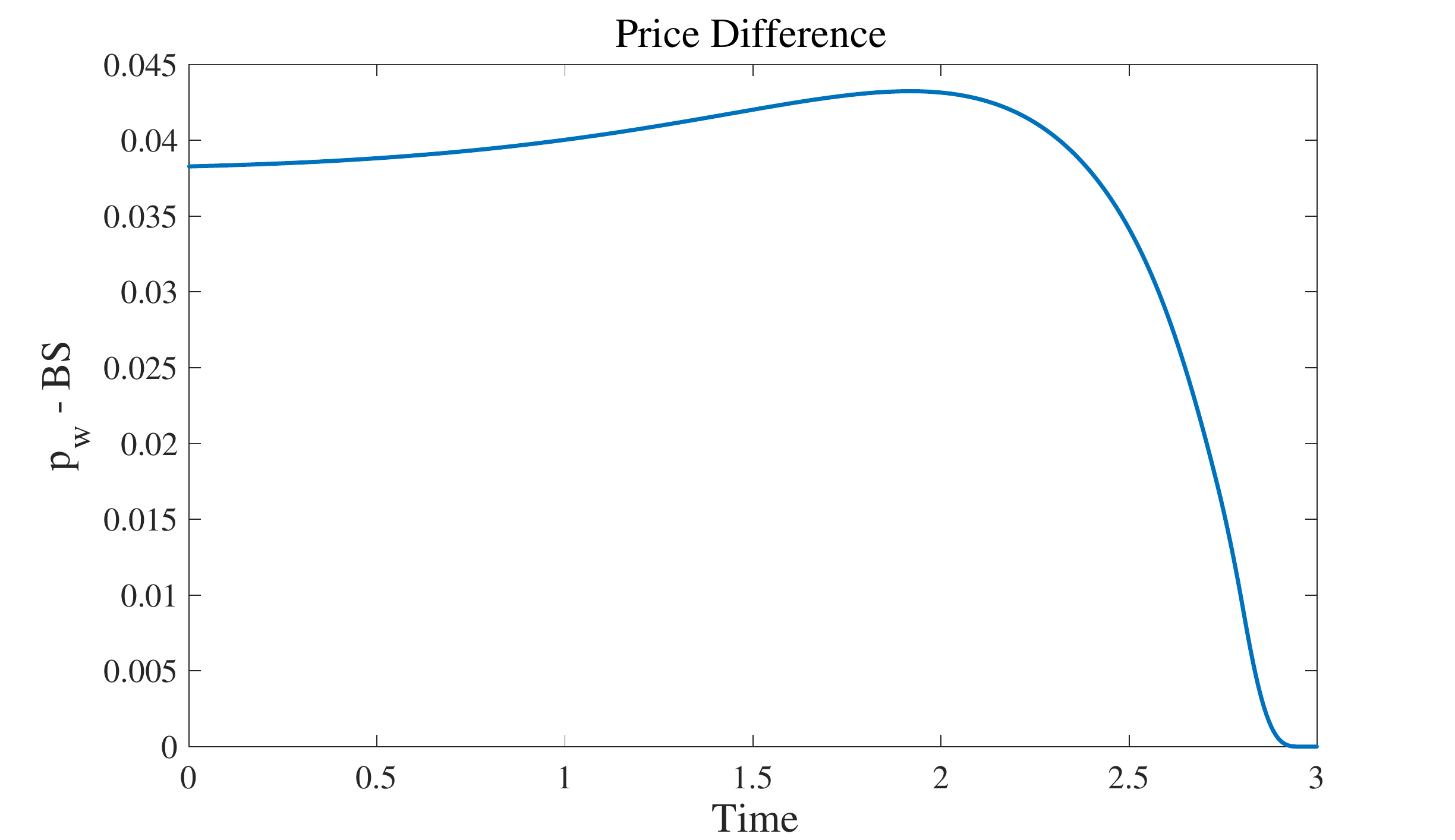}
\caption[Price difference obtained with the Pseudospectral method]{\label{Figura7} Price difference obtained with the Pseudospectral method. The results coincide with the numerical experiment in \cite{Davis2}.}
\end{figure}

We can observe that as $T\rightarrow\infty$, the price difference at $t=0$ approximates to $\lambda S$, the additional amount of money that is needed to purchase one share. This is empirically justified in \cite{Davis2} with a very natural interpretation: if maturity is big enough, it will be more likely that the option finishes In The money and it is exercised, so the seller will need to have one share.

This behaviour should repeat if we fix a maturity and compute option prices for the same strike but bigger stock prices. As we can compute now options as In The money as we want, this can be numerically checked.

\begin{figure}[h]
\centering
\includegraphics[width=7cm,height=4.5cm]{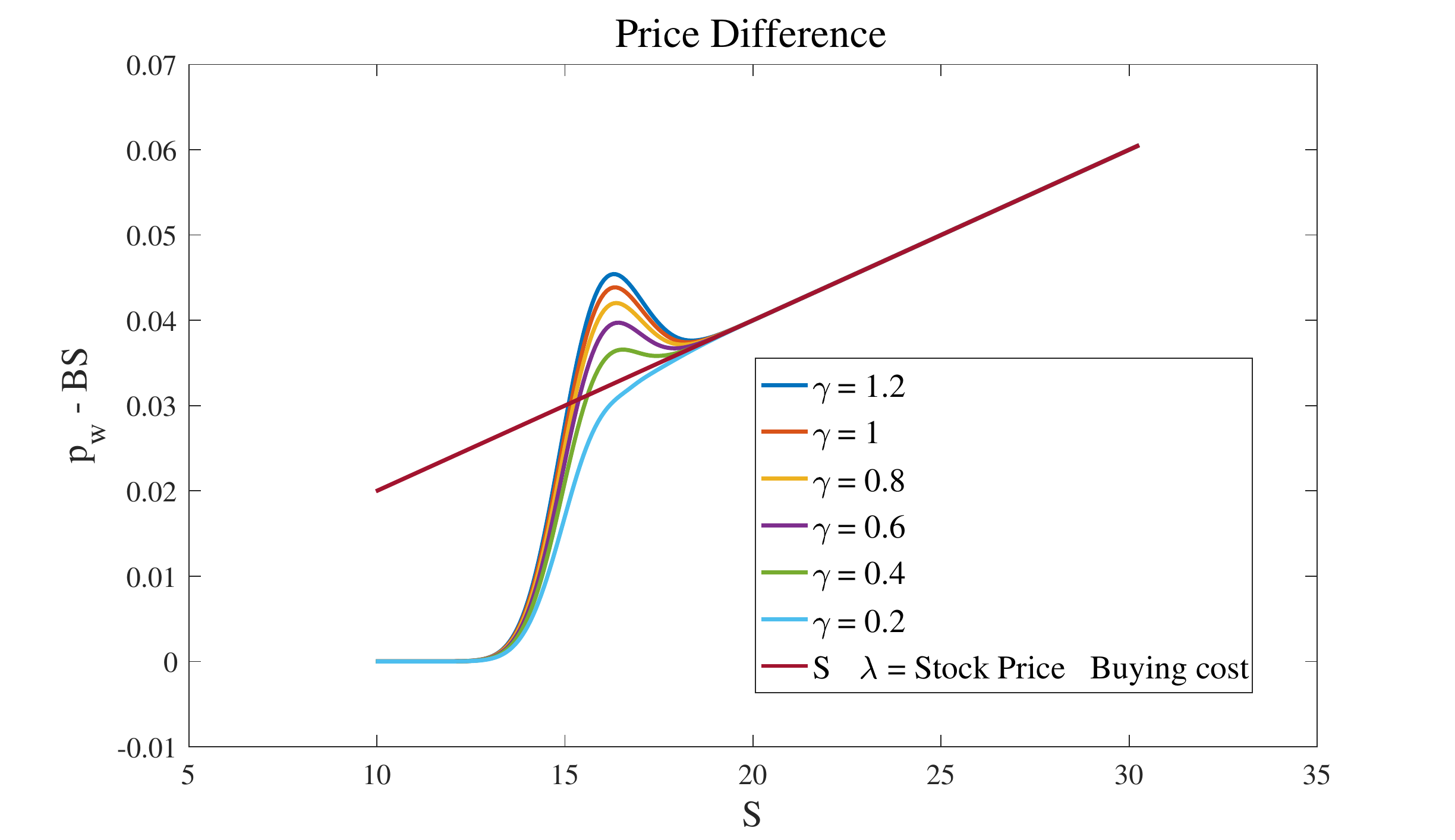}
\caption[Influence of the risk aversion parameter in the price difference]{\label{Ch4influencegamma} Influence of the risk aversion parameter in the price difference vs the stock price. }
\end{figure}

In Figure \ref{Ch4influencegamma} we see that $p_w-\text{BS}$ approximates to $\lambda S$ as $S\rightarrow \infty$. We can also check the influence of parameter $\gamma$, the index of risk aversion. As $\gamma$ grows, the seller of the option is more and more risk averse which means that he will demand more money to cover him from the possible transaction costs. , i.e. $S\rightarrow\infty \Rightarrow p_w-\text{BS}\rightarrow \lambda S.$

Concerning optimal trading strategies, it was conjectured in \cite{Davis2} that there exist two surfaces, which depend on $t$ and $S$, that lay up and below the optimal trading strategy when there were no transaction costs present. Numerical experiments seem to support this conjecture.

Other experiments were realized in \cite{Davis2} related with the ``overshoot'' ratio (OR), which is given by
\begin{equation}\label{Ch4overshootratio}
\text{OR}=\frac{(p_w-BS)-\lambda S}{\lambda S}.
\end{equation}

The evolution of OR in function of different parameter values can be studied. Some (empirical) properties of the overshoot ratio can be obtained \cite{Davis2}, as for example that the OR is linear increasing in function of $\log(\gamma)$. Figure \ref{Ch4statements1and2} represents the OR in function of $log(\gamma)$ and $S$ (left) and just in function of $S$ (right).

\begin{figure}[h]
\centering
\includegraphics[width=11cm,height=4.5cm]{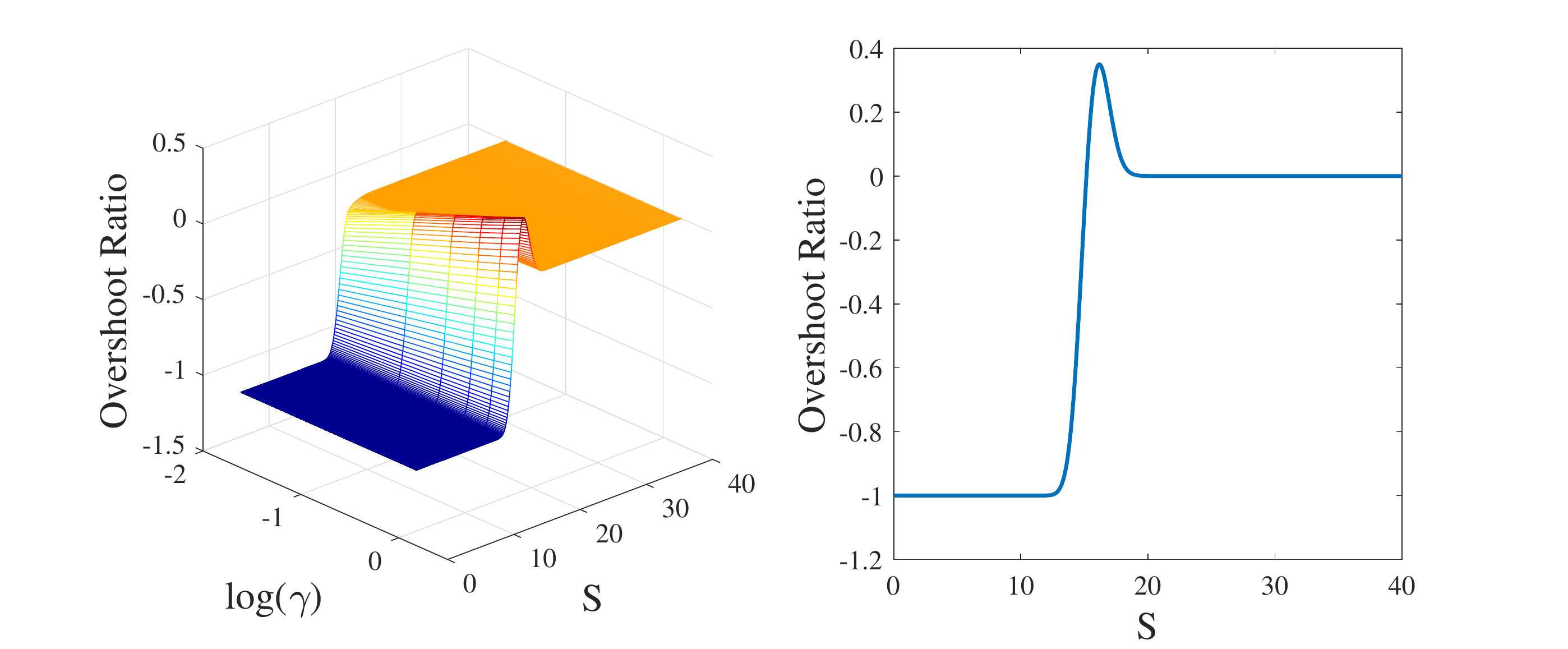}
\caption[Overshoot ratio dependance of $\log(\gamma)$ and $S$]{\label{Ch4statements1and2} Overshoot ratio dependance of $\log(\gamma)$ (left) and $S$ (right).}
\end{figure}

\newpage

\appendix
\appendixpage

\begin{proof}[Proof of Proposition \ref{Ch4equivalencia}] \

By the definition of $\mathscr{E}_{E}$, we know that the trading strategies
\begin{equation*}
\left\{
\begin{aligned}
& y^{\pi^{s}}\equiv 0, \quad j=1, \\
& y^{\pi^{s}}\equiv 1, \quad j=w,
\end{aligned}
\right.
\end{equation*}
are admissible in $\tau_{E}(X,y)$, so they are in  $\tau(X,y)$. Under these strategies, the final wealth satisfies
\begin{equation*}
W_j((T,X^{\pi^{s}}(T),y^{\pi^{s}}(T),\bar{S}(T))> - E, \quad j\in\{1,w\}.
\end{equation*}

Any trading strategy $\pi\in\tau$ that lies outside $\mathscr{E}_{E}$ for any $t\in[0,T]$, leads automatically to the residual utility $1-\exp(\gamma E)$, which is always suboptimal.

Therefore, the optimal trading strategy must belong to  $\tau_{E}(X,y)$. Consequently,
\begin{equation*}
V^{B_E}_j(t,X,y,S)=V^{\mathscr{E}_{E}}_j(t,X,y,S).
\end{equation*}
\end{proof}

\begin{proof}[Proof of Proposition \ref{Ch4Qbounded}] \

It is easy to check that functions $V^{B_E}_j, \ j\in\{1,w\}$ satisfy for $t\in[0,T]$ that:
\begin{equation*}
\begin{aligned}
V^{B_E}_j(t,X,y,S) & \geq 1-\exp(\gamma E), \quad (X,y,S)\in \mathscr{E}_{E}, \\
V^{B_E}_j(t,X,y,S) & = 1-\exp(\gamma E), \quad (X,y,S)\notin \mathscr{E}_{E},
\end{aligned}
\end{equation*}
where the first inequality is obtained by a suboptimality argument employing strategy $\pi^s$ of the previous proof and the second one comes from the definition of function $V^{B_E}_j, \ j\in\{1,w\}$.

The upper bound of Proposition \ref{Ch4Qbounded} is a consequence of formula (\ref{Ch4obtencQ}) and the previous inequalities.

By construction (see \cite[(4.22) and (4.25)]{Davis2}), function $Q^{B_{E_0,X_0}}_j, \quad j\in\{1,w\}$ is strictly positive.

\end{proof}

\begin{proof}[Proof of Theorem \ref{estabilidad}] \

By definition, we have for $j=0,1,...,2N-1$,
\small
\begin{equation*}
\begin{aligned}
\frac{\partial V^N}{\partial \tau}{(x_j,\tau)} &= A\frac{\partial^2 V^N}{\partial {x}^2}{(x_j,\tau)}+B\frac{\partial V^N}{\partial x}{(x_j,\tau)}+C\left(\frac{\partial V^N}{\partial {x}}{(x_j,\tau)}\right)^2+F^N(x_j,\tau),  \\
\frac{\partial W^N}{\partial \tau}{(x_j,\tau)} &= A\frac{\partial^2 W^N}{\partial {x}^2}{(x_j,\tau)}+B\frac{\partial W^N}{\partial x}{(x_j,\tau)}+C\left(\frac{\partial W^N}{\partial {x}}{(x_j,\tau)}\right)^2+G^N(x_j,\tau).
\end{aligned}
\end{equation*}
\normalsize

Subtracting both expressions, we obtain for $j=0,... \ ,2N-1$ and $\forall \tau\in[0,T]$:
\small
\begin{equation*}
\begin{aligned}
\frac{\partial e^N}{\partial \tau}{(x_j)}= & A\frac{\partial^2 e^N}{\partial {x}^2}{(x_j)}+B\frac{\partial e^N}{\partial x}{(x_j)}+C\left[\left(\frac{\partial V^N}{\partial {x}}{(x_j)}\right)^2-\left(\frac{\partial W^N}{\partial {x}}{(x_j)}\right)^2\right] \\ & +\left(F^N(x_j)-G^N(x_j)\right),
\end{aligned}
\end{equation*}
\normalsize
where $e^{N}=V^{N}-W^{N}$. Equivalently,
\begin{equation*}
\frac{\partial e^N}{\partial \tau}=A\frac{\partial^2 e^N}{\partial {x}^2}+B\frac{\partial e^N}{\partial x}+C\left(I_N\left[\left(\frac{\partial V^N}{\partial {x}}\right)^2-\left(\frac{\partial W^N}{\partial {x}}\right)^2\right]\right)+\left(F^N-G^N\right),
\end{equation*}
since, by definition, $e^N, \ F^N, \ G^N \in S_N$.

\

For $\phi \in S_N$, taking the scalar product of the previous expression with respect to $\phi$, we obtain:
\begin{equation*}\small
\begin{aligned}
\left(\frac{\partial e^N}{\partial \tau}, \ \phi \right) = & A\left(\frac{\partial^2 e^N}{\partial {x}^2}, \ \phi \right)+B\left(\frac{\partial e^N}{\partial x}, \ \phi \right)\\
& +C\left(I_N\left[\left(\frac{\partial V^N}{\partial {x}}{(x_j)}\right)^2-\left(\frac{\partial W^N}{\partial {x}}{(x_j)}\right)^2\right], \ \phi  \right) + \left(F^N-G^N,\phi\right).
\end{aligned}\normalsize
\end{equation*}

Taking $\phi=e^N$ and noting that the periodic boundary conditions imply that $\left(\frac{\partial e^N}{\partial x},e^N\right)=0$ and $\left(\frac{\partial^2 e^N}{\partial x^2},e^N\right) =-\left(\frac{\partial e^N}{\partial x},\frac{\partial e^N}{\partial x}\right)$ we get:
\begin{equation*}
\begin{aligned}
\frac{1}{2}\frac{d }{d\tau}||e^N(\tau)||^2 +A ||(e^N)_x(\tau)||^2 =& C\left(I_N\left[\left(\frac{\partial V^N}{\partial {x}}\right)^2-\left(\frac{\partial W^N}{\partial {x}}\right)^2\right], \ e^N  \right) \\ & + \left(F^N-G^N,e^N\right).
\end{aligned}
\end{equation*}

Since for any pair of functions $u,v\in S_{N}$, it holds $(u,v)_N=(u,v)$, where $(u,v)_N$ denotes the usual discrete scalar product (see \cite[(2.1.33)]{Canuto}), we have  \small
\begin{equation*}
\begin{aligned}
& \left|\left(I_N\left[\left(\frac{\partial V^N}{\partial {x}}\right)^2-\left(\frac{\partial W^N}{\partial {x}}\right)^2\right], \ e^N  \right)\right| \\
& \ \ \ = \left|\sum_{j=0}^{2N-1}\left[\left(\frac{\partial V^N}{\partial {x}}(x_j)\right)^2-\left(\frac{\partial W^N}{\partial {x}}(x_j)\right)^2\right] e^N(x_j) \right|  \\
& \ \ \ \leq \sum_{j=0}^{2N-1} \left|\left[\left(\frac{\partial V^N}{\partial {x}}(x_j)\right)^2-\left(\frac{\partial W^N}{\partial {x}}(x_j)\right)^2\right]\right|\left|e^N(x_j)\right| \\
& \ \ \ \leq \sum_{j=0}^{2N-1} \left|\left(\frac{\partial V^N}{\partial {x}}(x_j)\right)+\left(\frac{\partial W^N}{\partial {x}}(x_j)\right)\right| \left|\left(\frac{\partial V^N}{\partial {x}}(x_j)\right)-\left(\frac{\partial W^N}{\partial {x}}(x_j)\right)\right|\left|e^N(x_j)\right| \\
& \ \ \ \leq 2M \sum_{j=0}^{2N-1} \left|(e^N)_x(x_j)\right|\left|e^N(x_j)\right| \\
& \ \ \ \leq 2M \|(e^N)_x\|_N\|e^N\|_N=2M\|(e^N)_x\|\|e^N\|,
\end{aligned}
\end{equation*} \normalsize
where, from the hypothesis of the theorem, we have employed:
\begin{equation*}
\left\|\frac{\partial V^N}{\partial x}+\frac{\partial W^N}{\partial x}\right\|_{\infty} \leq  \left\|\frac{\partial V^N}{\partial x}\right\|_{\infty}+\left\|\frac{\partial W^N}{\partial x}\right\|_{\infty} \leq 2M,
\end{equation*}

Therefore, we can bound
\begin{equation*}
\begin{aligned}
&\left|C\left(I_N\left[\left(\frac{\partial V^N}{\partial {x}}\right)^2-\left(\frac{\partial W^N}{\partial {x}}\right)^2\right], \ e^N  \right)\right|\leq 2M|C|\|(e^N)_x\|\|e^N\|.
\end{aligned}
\end{equation*}

Using Cauchy Schwartz's inequality to bound $\left(F^N-G^N,e^N\right)$, we get
\begin{equation*}
\begin{aligned}
\frac{1}{2}\frac{d }{d\tau}||e^N(\tau)||^2 +A ||(e^N)_x(\tau)||^2 \leq & 2M|C|\|(e^N)_x(\tau)\|\|e^N(\tau)\| \\ & +\|F^N(\tau)-G^N(\tau)\|\|e^N(\tau)\|.
\end{aligned}
\end{equation*}

We apply inequality $ab\leq \left(\epsilon a^2+\frac{1}{4\epsilon}b^2\right), \ \ a,b>0$, to both terms on the right side, using respectively $\epsilon=\frac{A}{4M|C|}$ and $\epsilon=1$.
\begin{equation*}
\begin{aligned}
\frac{1}{2}\frac{d }{d\tau}||e^N(\tau)||^2 +A\|(e^N)_x(\tau)||^2 \leq & \frac{A}{2}\|(e^N)_x(\tau)\|^2+ \frac{2M^2C^2}{A}\|e^N(\tau)\|^2 \\
& +\|F^N(\tau)-G^N(\tau)\|^2+\frac{1}{4}\|e^N(\tau)\|^2,
\end{aligned}
\end{equation*}
so that
\begin{equation*}
\frac{1}{2}\frac{d }{d\tau}||e^N(\tau)||^2 +\frac{A}{2}\|(e^N)_x(\tau)\|^2\leq K \|e^N(\tau)\|^2+\|F^N(\tau)-G^N(\tau)\|^2,
\end{equation*}
where $K=\frac{2M^2C^2}{A}+\frac{1}{4}$.

\

Using Gronwall's lemma (see \cite[A.15]{Canuto}), \small
\begin{equation*}
\underset{0\leq \tau\leq T}{\max}\|e^N(\tau)\|^2+\frac{A}{2}\int_0^T\|(e^N)_x(\tau)\|^2d\tau \leq R\left(\|e^N(0)\|^2+\int_0^T\|F^N(\tau)-G^N(\tau)\|^2d\tau\right),
\end{equation*}\normalsize
with $R=\exp(KT)$.

\end{proof}

\begin{proof}[Proof of Proposition \ref{Ch4lemmainterpderiv}]

\

For any $\tau\in[0, T]$, we decompose:
\begin{equation}\label{Ch4emmainterpderbisss}
\begin{aligned}
\left\|\frac{\partial^s P_N(u)}{\partial x^s}\right\|_{\infty} &\leq \left\|\frac{\partial^s P_N(u)}{\partial x^s}-\frac{\partial^s u}{\partial x^s}\right\|_{\infty}+\left\|\frac{\partial^s u}{\partial x^s}\right\|_{\infty}, \\
\left\|\frac{\partial^s I_N(u)}{\partial x^s}\right\|_{\infty} &\leq \left\|\frac{\partial^s I_N(u)}{\partial x^s}-\frac{\partial^s u}{\partial x^s}\right\|_{\infty}+\left\|\frac{\partial^s u}{\partial x^s}\right\|_{\infty}.
\end{aligned}
\end{equation}

Inequality \cite[(A.12)]{Canuto} implies that
\begin{equation*}
\left\|\frac{\partial^s u}{\partial x^s}\right\|_{\infty} \leq C_1 \left\|\frac{\partial^{s+1} u}{\partial x^{s+1}}\right\|_{L^2},
\end{equation*}
and
\begin{equation*}
\begin{aligned}
\left\|\frac{\partial^s P_N(u)}{\partial x^s}-\frac{\partial^s u}{\partial x^s}\right\|_{\infty} & \leq C_1 \left\|\frac{\partial^s P_N(u)}{\partial x^s}-\frac{\partial^s u}{\partial x^s}\right\|_{H^1}, \\
\left\|\frac{\partial^s I_N(u)}{\partial x^s}-\frac{\partial^s u}{\partial x^s}\right\|_{\infty} &\leq C_1 \left\|\frac{\partial^s I_N(u)}{\partial x^s}-\frac{\partial^s u}{\partial x^s}\right\|_{H^1},
\end{aligned}
\end{equation*}

\

Applying \cite[(5.1.5)]{Canuto} (Bernstein's inequality), standard approximation results of projection \cite[(5.1.10)]{Canuto} and aliasing error $\left(\|I_N(u)-P_N(u)\|_{L^2}\right)$ result \cite[(5.1.18)]{Canuto}, we can bound
\begin{equation*}
\begin{aligned}
\left\|\frac{\partial^s P_N(u)}{\partial x^s}-\frac{\partial^s u}{\partial x^s}\right\|_{H^1} &\leq K_1 N^{1-r}\left\|\frac{\partial^{s+r} u}{\partial x^{s+r}}\right\|_{L^2}, \\
\left\|\frac{\partial^s I_N(u)}{\partial x^s}-\frac{\partial^s u}{\partial x^s}\right\|_{H^1} & \leq \left\|\frac{\partial^s I_N(u)}{\partial x^s}-\frac{\partial^s P_N(u)}{\partial x^s}\right\|_{H^1}+\left\|\frac{\partial^s P_N(u)}{\partial x^s}-\frac{\partial^s u}{\partial x^s}\right\|_{H^1}  \\
&\leq N^s\|I_N(u)-P_N(u)\|_{H^1}+K_1 N^{1-r}\left\|\frac{\partial^{s+r} u}{\partial x^{s+r}}\right\|_{L^2}  \\
&\leq N^{s+1}\|I_N(u)-P_N(u)\|_{L^2}+K_1 N^{1-r}\left\|\frac{\partial^{s+r} u}{\partial x^{s+r}}\right\|_{L^2}  \\
& \leq K_1N^{1-r}\left\|\frac{\partial^{s+r} u}{\partial x^{s+r}}\right\|_{L^2}+K_1 N^{1-r}\left\|\frac{\partial^{s+r} u}{\partial x^{s+r}}\right\|_{L^2} \\
&=2K_1N^{1-r}\left\|\frac{\partial^{s+r} u}{\partial x^{s+r}}\right\|_{L^2},
\end{aligned}
\end{equation*}
if $u\in H^{s+r}_p, \ r\geq 1$.

The choice of
\begin{equation*}
M=(2K_1+C_1)\underset{0\leq\tau\leq T}{\max}\left\|\frac{\partial^{s+1} u}{\partial x^{s+1}}\right\|,
\end{equation*}
completes the proof.
\end{proof}

\begin{proof}[Proof of Proposition \ref{Ch4Consistproposition}]

\

Let us define the function:
\begin{equation}\label{Ch4consisteneq1}
J^{2N}=\frac{\partial I_N(u(x,\tau))}{\partial \tau}-A\frac{\partial^2 I_N(u(x,\tau))}{\partial {x}^2}-B\frac{\partial I_N(u(x,\tau))}{\partial x}-C\left[\left(\frac{\partial I_N(u(x,\tau))}{\partial {x}}\right)^2\right],
\end{equation}
and note that:
\begin{equation*}
\begin{aligned}
J^{2N} &\in S_{2N}, \\
F^N &= I_N\left(J^{2N}\right).
\end{aligned}
\end{equation*}

The function $u(x,\tau)$ satisfies:
\begin{equation}\label{Ch4consisteneq2}
0=\frac{\partial u(x,\tau)}{\partial \tau} -A\frac{\partial^2 u(x,\tau)}{\partial {x}^2}-B\frac{\partial u(x,\tau)}{\partial x}-C\left[\left(\frac{\partial u(x,\tau)}{\partial {x}}\right)^2\right].
\end{equation}

Subtracting (\ref{Ch4consisteneq1}) and (\ref{Ch4consisteneq2}):
\begin{equation*}
J^{2N}=J^{2N}_1-AJ^{2N}_2-BJ^{2N}_3-CJ^{2N}_4,
\end{equation*}
with
\begin{equation*}
\begin{aligned}
& J^{2N}_1=\frac{\partial I_N(u(x,\tau))}{\partial \tau}-\frac{\partial u(x,\tau)}{\partial \tau}, \\
& J^{2N}_2=\frac{\partial^2 I_N(u(x,\tau))}{\partial {x}^2}-\frac{\partial^2 u(x,\tau)}{\partial {x}^2}, \\
& J^{2N}_3=\frac{\partial I_N(u(x,\tau))}{\partial x}-\frac{\partial u(x,\tau)}{\partial x}, \\
& J^{2N}_4=\left(\frac{\partial I_N(u(x,\tau))}{\partial {x}}\right)^2-\left(\frac{\partial u(x,\tau)}{\partial {x}}\right)^2,
\end{aligned}
\end{equation*}
for all $\tau \in [0,T]$.

\

The no linear term $J^{2N}_4$ is bounded by:
\begin{equation*}
\left\|\left(\frac{\partial I_N(u)}{\partial {x}}\right)^2-\left(\frac{\partial u}{\partial {x}}\right)^2\right\|  \leq \left\|\frac{\partial I_N(u)}{\partial {x}}+\frac{\partial u}{\partial {x}}\right\|_{\infty}\left\|\frac{\partial I_N(u)}{\partial {x}}-\frac{\partial u}{\partial {x}}\right\|.
\end{equation*}

From Proposition \ref{Ch4lemmainterpderiv}, it exists a constant $M_1=M_1\left(\underset{0\leq\tau\leq T}{\max}\left\|\frac{\partial^{2} u}{\partial x^{2}}\right\|\right)$ such that
\begin{equation*}
\underset{0\leq \tau \leq T}{\max}\left\{\left\|\frac{\partial I_N(u)}{\partial {x}}{(\tau)}\right\|_{\infty},\left\|\frac{\partial u}{\partial {x}}{(\tau)}\right\|_{\infty}\right\}\leq M_1,
\end{equation*}

The second term is bounded by
\begin{equation*}
\left\|u_x-\left(I_Nu\right)_x\right\|_{L^2}\leq K_1N^{-s}\left\|\frac{\partial^{s+1} u(x,\tau)}{\partial {x}^{s+1}}\right\|_{L^2},
\end{equation*}
due to the approximation result \cite[(5.1.20)]{Canuto}.

\

Therefore, the no linear term $J^{2N}_4$ is bounded by
\begin{equation*}
\left\|\left(\frac{\partial I_N(u)}{\partial {x}}\right)^2-\left(\frac{\partial u}{\partial {x}}\right)^2\right\|  \leq K_1M_1N^{-s}\left\|\frac{\partial^{s+1} u}{\partial x^{s+1}}\right\|.
\end{equation*}

Obviously, term $J^{2N}_3$ can also be bounded by \cite[(5.1.20)]{Canuto}:
\begin{equation*}
\left\|J^{2N}_3\right\|\leq K_1N^{-s}\left\|\frac{\partial^{s+1} u}{\partial {x}^{s+1}}\right\|.
\end{equation*}

Using again approximation result \cite[(5.1.9)]{Canuto}, Bernstein's inequality \cite[(5.1.5)]{Canuto} and aliasing error result \cite[(5.1.18)]{Canuto}, term $J^{2N}_2$ can be bounded by
\begin{equation*}
\begin{aligned}
\|J^{2N}_2\| & \leq  \left\|\frac{\partial^2 I_N(u)}{\partial {x}^2}-\frac{\partial^2 P_N(u)}{\partial {x}^2} \right\| + \left\|\frac{\partial^2 P_N(u(x,t))}{\partial {x}^2}-\frac{\partial^2 u(x,t)}{\partial {x}^2} \right\| \\
& \leq  N^2\left\|I_N(u)-P_N(u)\right\|+\left\|P_N(u_{xx})-u_{xx}\right\|  \\
& \leq  K_1N^{-s}\left\|\frac{\partial^{s} u_{xx}}{\partial {x}^{s}}\right\|+K_1N^{-s}\left\|\frac{\partial^{s} u_{xx}}{\partial {x}^{s}}\right\| = 2K_1N^{-s}\left\|\frac{\partial^{s} u_{xx}}{\partial {x}^{s}}\right\|.
\end{aligned}
\end{equation*}

For the last term, using \cite[(5.1.16)]{Canuto}:
\begin{equation*}
\|J^{2N}_1\| = \left\|I_N\left(\frac{\partial u(x,\tau)}{\partial \tau}\right)-\frac{\partial u(x,\tau)}{\partial \tau}\right\| \leq K_1N^{-s}\left\|\frac{\partial^{s} u_\tau}{\partial {x}^{s}}\right\|,
\end{equation*}
since interpolation does commute with derivation with respect the temporal variable.

\

Thus, depending on $\underset{0\leq \tau\leq T}{\max}\left\{ \left\|\frac{\partial^{s+1} u}{\partial x^{s+1}}\right\|, \left\|\frac{\partial^{s+2} u}{\partial x^{s+2}}\right\| , \ \left\|\frac{\partial^{s} u_{\tau}}{\partial x^{s}}\right\| \right\}$, there exists a constant $M_2\geq 0$, such that:
\begin{equation*}
\|J^{2N}\|\leq M_2N^{-s}.
\end{equation*}

Finally,  using  again Bernstein's inequality \cite[(5.1.5)]{Canuto} and the approximation result \cite[(5.1.16)]{Canuto}, since $J^{2N}\in S_{2N}$ we can bound
\begin{equation*}
\begin{aligned}
\left\|F^N\right\| &\leq \left\|F^N-J^{2N}\right\| + \left\|J^{2N}\right\| = \left\|I_N\left(J^{2N}\right)-J^{2N}\right\| + \left\|J^{2N}\right\| \\ & \leq K_1N^{-s} \left\|\frac{\partial^s J^{2N}}{\partial x^s}\right\| + \left\|J^{2N}\right\|
 \leq K_1 2^s \left\|J^{2N}\right\| + \left\|J^{2N}\right\| \\ & \leq (K_1 \cdot 2^s+1)\left\|J^{2N}\right\| \leq MN^{-s},
\end{aligned}
\end{equation*}
where $M=\left(K_1 2^{s}+1\right)M_2$.

\end{proof}

\begin{proof}[Proof of Proposition \ref{threshold}] \

Let
\begin{equation*}
M_1=\underset{0\leq\tau\leq T}{\max}\left\|(I_Nu(\tau))_x\right\|_{\infty},
\end{equation*}
which exists from Proposition \ref{Ch4lemmainterpderiv} under the regularity hypothesis of $u$.

\

For $\tau=0$ we have that $\|(u^N)_x(0)\|_{\infty}=\left\|(I_Nu(0))_x\right\|_{\infty}\leq M_1$.

\

By a continuity argument,  it must exist $\epsilon>0$  and $N_1$ big enough such that $\forall N\geq N_1$ it holds that
\begin{equation}\label{Ch4thrcondt}
\|u^N_x(\tau)\|_{\infty} \leq 2M_1, \ \ t\in[0,\epsilon].
\end{equation}

We argue by contradiction. For any $N\in\mathbb{N}$, we define:
\begin{equation*}
\epsilon_N=\underset{\tau}{\sup} \left\{0<\tau\leq T : \|u^{N}_x(s)\|_{\infty}< 2M_1, \ s\in[0,\tau]\right\}.
\end{equation*}
where it holds that $\epsilon_{N}>0$ because $u^N$ is the solution of an ODE system.

\

If (\ref{Ch4thrcondt}) does not hold, we can find a strictly increasing sequence $N_{1}, N_{2},... \rightarrow \infty$ and a strictly decreasing sequence $\epsilon_{N_1}, \epsilon_{N_2}... \rightarrow 0$ such that
\begin{equation}\label{Ch4contradiction1}
\underset{n\rightarrow\infty}{\lim} \ \underset{0\leq\tau \leq \epsilon_{N_n}}{\max} \|u^{N_n}_{x}(\tau)\|_{\infty}=2M_1.
\end{equation}

Applying Nicholsky and Bernstein inequalities,
\begin{equation*}
\begin{aligned}
\|u^{N_n}_{x}(\tau)\|_{\infty}& \leq \|u^{N_n}_{x}(\tau)-\left(I_N(u(\tau))\right)_x\|_{\infty}+\|\left(I_{N_n}u\right)_x(\tau)\|_{\infty} \\
& \leq K_1N_n^{\frac{3}{2}}\left\|u^{N_n}(\tau)-I_{N_n}u(\tau)\right\|+\|\left(I_{N_n}u\right)_x(\tau)\|_{\infty} \\
& \leq K_1N_n^{\frac{3}{2}}\left\|u^{N_n}(\tau)-I_{N_n}u(\tau)\right\|+M_1.
\end{aligned}
\end{equation*}

\

By construction, $\|u^{N_n}_x(\tau)\|\leq 2M_1, \ \tau\in[0,\epsilon_{N_n}]$,  therefore, employing the arguments used in the proof of the stability Theorem \ref{estabilidad} with $V^{N_n}=I_{N_n}u$ and $W^{N_n}=u^{N_n}$, it holds: \small
\begin{equation*}
\underset{0\leq \tau \leq \epsilon_{N_n}}{\max}\left\|u^{N_n}(\tau)-I_{N_n}u(\tau)\right\|^2\leq R\left(\|I_{N_n}(u(0))-u^{N_n}(0)\|^2+ \int_0^{\epsilon_{N_n}} \|F^{N_n}(\tau)\|^2 d\tau \right),
\end{equation*} \normalsize
where $\|I_{N_n}(u(0))-u^{N_n}(0)\|^2=0$ by definition of the collocation method and term $\|F^{N_n}(\tau)\|$ is given by (\ref{Ch4exprthconsistencia}).

\

Term $\|F^{N_n}(\tau)\|$ can be globally bounded in $[0,T]$. Therefore in $[0,\epsilon_{N_n}]$, by Proposition \ref{Ch4Consistproposition} and the regularity hypothesis over $u$
\begin{equation*}
\|F^{N_n}(\tau)\|\leq M_2N_n^{-s} \leq M_2N_n^{-2},
\end{equation*}

This implies, rearranging terms, that for any $\tau\in[0,\epsilon_{N_n}]$
\begin{equation}\label{Ch4contradiction2}
\underset{0\leq\tau \leq \epsilon_{N_n}}{\max} \|u^{N_n}_{x}(\tau)\|_{\infty} \leq K N_n^{-\frac{1}{2}}+M_1,
\end{equation}
where $K$ is a constant that depends on $M_2$ and $R$. This is a contradiction with (\ref{Ch4contradiction1}) since
\begin{equation}
\underset{n\rightarrow\infty}{\lim} \ \underset{0\leq\tau \leq \epsilon_{N_n}}{\max} \|u^{N_n}_{x}(\tau)\|_{\infty}\leq M_1<2M_1.
\end{equation}

\

Now, let $\epsilon^{*}>0$ be the maximum value for which it exists a value $N_0$ big enough such that $\forall N\geq N_0$ it holds that
\begin{equation*}
\|u^{N}_{x}(\tau)\|_{\infty} < 2M_1, \quad \tau\in[0,\epsilon^{*}].
\end{equation*}

Suppose $\epsilon^{*}<T$. We argue as before, so that
\begin{equation*}
\|u^{N}_{x}(\tau)\|_{\infty} \leq K_1N_n^{\frac{3}{2}}\left\|u^{N_n}(\tau)-I_{N_n}u(\tau)\right\|+M_1.
\end{equation*}
and noting that $\forall N\geq N_0$, by Stability Theorem \ref{estabilidad} on $[0, \epsilon^{*}]$
\begin{equation}\label{Ch4contradiction3}
\underset{0\leq\tau \leq \epsilon^{*}}{\max} \|u^{N_n}_{x}(\tau)\|_{\infty} \leq K N^{-\frac{1}{2}}+M_1,
\end{equation}

Again, by a continuity argument, there must exist $\epsilon^{*}_1>\epsilon^{*}$ and a value $N^{\epsilon^{*}}_0>$ big enough, such that $\forall N \geq N^{\epsilon^{*}}_0$ it holds that $\|(u^N)_x\|_{\infty} <  2M_1, \quad  \tau\in[0,\epsilon^{*}_1]$. Otherwise we could find a strictly increasing sequence $N^{\epsilon^{*}}_{1}, N^{\epsilon^{*}}_{2},... \rightarrow \infty$ and a strictly decreasing sequence $\epsilon_{N^{\epsilon^{*}}_1}, \epsilon_{N^{\epsilon^{*}}_2}... \rightarrow \epsilon^{*}$ such that
\begin{equation*}
\underset{n\rightarrow\infty}{\lim} \ \underset{0\leq\tau \leq \epsilon_{N^{\epsilon^{*}}_n}}{\max} \|u^{N^{\epsilon^{*}}_n}_{x}(\tau)\|_{\infty}=2M_1,
\end{equation*}
which would lead to a contradiction with (\ref{Ch4contradiction3}) exactly with same arguments as before.

\end{proof}

\begin{proof}[Proof of Theorem \ref{convergencia}]

\

We decompose:
\begin{equation*}
\|u(\tau)-u^{N}(\tau)\| \leq \|u(\tau)-I_N(u(\tau)\|+\|I_N(u(\tau))-u^{N}(\tau)\|.
\end{equation*}

The term $\|u(\tau)-I_N(u(\tau)\|$ is bounded by the estimate \cite[(5.1.16)]{Canuto}
\begin{equation*}
\underset{0\leq \tau \leq T}{\max}\|u(\tau)-I_N(u(\tau))\|\leq K_1N^{-s}\underset{0\leq \tau \leq T}{\max}\left\|\frac{\partial^s u}{\partial x^s}(\tau)\right\|.
\end{equation*}

We apply Theorem \ref{estabilidad} to the second term $\|I_N(u(\tau))-u^{N}(\tau)\|$, taking $V^N=I_N(u)$ and $W^N=u^{N}$. Note that the definition of the collocation method (\ref{Ch4fourcolmeth}) implies that $G^N\equiv 0$ and that the threshold condition (\ref{Ch4threshcondstab}) holds for $N\geq N_0$ big enough from Proposition \ref{threshold}. Therefore,
\begin{equation*}
\underset{0\leq \tau \leq T}{\max}\|I_N(u(\tau))-u^N(\tau)\|^2 \leq R\left(\|I_N(u(0))-u^N(0)\|^2+ \int_0^T \|F^N(\tau)\|^2 d\tau \right).
\end{equation*}

We apply Proposition \ref{Ch4Consistproposition} to bound
\begin{equation*}
\|F^N(\tau)\|\leq M_1N^{-s}.
\end{equation*}

For completing the proof, note that in the collocation method $u^N(0)=I_N(u_0)$. Therefore, we can bound
\begin{equation*}
\begin{aligned}
\underset {0\leq \tau \leq T}{\max} \left\{\|u(\tau)-u^{N}(\tau)\|\right\} \leq K_1N^{-s}\underset{0\leq \tau \leq T}{\max}\left\|\frac{\partial^s u}{\partial x^s}(\tau)\right\|+\sqrt{R}M_1N^{-s},
\end{aligned}
\end{equation*}
by the regularity hypothesis over $u$.

\end{proof}

\begin{proof}[Proof of Proposition \ref{penulprop}]  \

Function $f^e$ admits a classical derivative in $(0,2\pi)$
\begin{equation*}
(f^e(x))'=
\left\{
\begin{aligned}
&f'(x), && x\in \left(0,\frac{\pi}{2}\right], \\
&f'\left(\pi-x\right), && x\in \left(\frac{\pi}{2}, \pi \right], \\
&-f'\left(x-\pi\right), && x\in \left(\pi,\frac{3\pi}{2} \right], \\
&-f'\left(2\pi-x\right), && x\in \left(\frac{3\pi}{2},2\pi \right), \\
\end{aligned}
\right.
\end{equation*}
because $f'(0^{-})=0 \Rightarrow (f^e)'(\pi-\pi^{-})=-(f^e)'(\pi^{+}-\pi)=0$. It also admits a second derivative (in distributional sense)
\begin{equation*}
(f^e(x))''=
\left\{
\begin{aligned}
&f''(x), && x\in \left(0,\frac{\pi}{2}\right), \\
&-f''\left(\pi-x\right), && x\in \left(\frac{\pi}{2}, \pi \right), \\
&-f''\left(x-\pi\right), && x\in \left(\pi,\frac{3\pi}{2} \right), \\
&f''\left(2\pi-x\right), && x\in \left(\frac{3\pi}{2},2\pi \right), \\
\end{aligned}
\right.
\end{equation*}
defined everywhere but for $x=\left\{\frac{\pi}{2},\pi,\frac{3\pi}{2}\right\}$.

\

Therefore, $f^e\in H^{2}_p$ and the standard approximation result for interpolation and Proposition \ref{Ch4lemmainterpderiv} can be applied.
\end{proof}

\begin{proof}[Proof of Proposition \ref{ultimprop}] \

For $u_0=u_w(T,y_k,x)$ it is easy to check that $u_0\in H^1_p$, so let us study $u_{0,x}$. Function $u_{0,x}$ is of finite variation, derivable everywhere except at two points where it presents two jump discontinuities and which correspond to the strike value up to the odd-even extension and the change of variable.

\

In this case, we know that the truncated Fourier series $P_N(u_{0,x})$ converges pointwise to $\frac{u_{0,x}(x^{-})+u_{0,x}(x^{+})}{2}$. Therefore, it exists $C$, independent of $N$, such that $\|P_N(u_{0,x})\|_{\infty}\leq C$ (see analysis of the Gibbs effect in \cite{Canuto}).

\

We perform the decomposition
\begin{equation*}
\begin{aligned}
\|(I_N(u_0))_{x}\|_{\infty} & \leq \|(I_N(u_0))_{x}-(P_N(u_0))_x\|_{\infty} +\|(P_N(u_0))_{x}\|_{\infty} \\
& \leq K_1N^{\frac{3}{2}}\|I_N(u_0)-P_N(u_0)\|+\|P_N(u_{0,x})\|_{\infty},
\end{aligned}
\end{equation*}
where we have used Bernstein and Nicholsky inequalities and the fact that truncation does permute with differentiation.

\

Now, since $u_0\in H^1_p$ it holds that $\|I_N(u_0)-P_N(u_0)\|\leq K_2 \|I_N(u_0)-u_0\|$. Therefore, we can bound
\begin{equation*}
\|(I_N(u_0))_{x}\|_{\infty} \leq KN^{\frac{3}{2}}\|I_N(u_0)-u_0\|+C.
\end{equation*}
\end{proof}

\begin{proof}[Proof of Theorem \ref{ultimtheor}] \

 Note that it holds $\forall \hat{x}\in[-L,L]$
\begin{equation*}
R^{\hat{x}^{*}}_p(\hat{x},t)-R(\hat{x},t)=\int_{-\infty}^{\infty} \left(\phi^{\hat{x}^{*}}_p(\hat{x}')-\phi(\hat{x}')\right)\Theta(\hat{x}',\hat{x},t,t_0) d\hat{x}',
\end{equation*}
where $\Theta(\hat{x}',\hat{x},t,t_0)=\frac{1}{\sigma\sqrt{2\pi(t_0-t)}}\exp\left(\frac{-\left[\hat{x}'-(\hat{x}+\left(\alpha-\frac{\sigma^2}{2}\right)(t_0-t))\right]^2}{2\sigma^2(t_0-t)}\right)$.

\

This function can be split in :
\begin{equation*}
R^{\hat{x}^{*}}_p(\hat{x},t)-R(\hat{x},t)=\int_{-\infty}^{-\hat{x}^{*}} \left(\phi^{\hat{x}^{*}}_p(\hat{x}')-\phi(\hat{x}')\right)\Theta d\hat{x}'+\int_{\hat{x}^{*}}^{\infty} \left(\phi^{\hat{x}^{*}}_p(\hat{x}')-\phi(\hat{x}')\right) \Theta d\hat{x}',
\end{equation*}
because, by construction, $\phi^{\hat{x}^{*}}_p(\hat{x}')=\phi(\hat{x}'), \ \hat{x}'\in[-\hat{x}^{*},\hat{x}^{*}]$.

\

By Lemma \ref{Ch4Qbounded}, $0\leq\phi(\hat{x}')\leq M$, and this implies, by construction, $0\leq \phi^{\hat{x}^{*}}_p(\hat{x}')\leq 2M$. Therefore, we can bound
\begin{equation*}
\left|R^{\hat{x}^{*}}_p(\hat{x},t)-R(\hat{x},t)\right|=3M\int_{-\infty}^{-\hat{x}^{*}}\Theta d\hat{x}'+3M\int_{\hat{x}^{*}}^{\infty} \Theta d\hat{x}'.
\end{equation*}

The result of the theorem is now straightforward since it is well known (see \cite{Frutos}) that
\begin{equation*}
\begin{aligned}
\int_{-\infty}^{-\hat{x}^{*}}\Theta d\hat{x}' & \underset{-\hat{x}^{*}\rightarrow -\infty}{\longrightarrow} 0,  \\
\int_{\hat{x}^{*}}^{\infty}\Theta d\hat{x}' & \underset{\hat{x}^{*}\rightarrow \infty}{\longrightarrow} 0.
\end{aligned}
\end{equation*}

\end{proof}


\begin{thebibliography}{99}

\bibitem{Black} Black F., Scholes M., {\em The Pricing of Options and Corporate Liabilities\/}, The Journal of Political Economy, 81 (1973), 637-654.

\bibitem{BME} BME Clearing House {\em Central counterparty entity Regulations. \/}

\bibitem{Breton} Breton M. and de Frutos J., {\em Option Pricing under GARCH Processes by PDE Methods\/},  Operations Research, 58 (2010), 1148-1157.

\bibitem{Breton2} Breton M. and de Frutos J., {\em Approximation of Dynamic Programs},  in Handbook of Computational Finance, 633-649, Jin-Chuan Duan, James E. Gentle, and Wolfgang H\"{a}rdle(eds), Springer, 2012.

\bibitem{Canuto} Canuto C., Hussaini  M.Y., Quarteroni A. and Zang T.A., {\em Spectral Methods. Fundamentals in single domains\/}, Springer, Berlin, 2006.

\bibitem{Carmona} Carmona R., {\em Indifference Pricing\/}, Princeton University Press, Princeton, 2009.


\bibitem{Chiarella} Chiarella C., El-Hassan N., and Kucera A., {\em Evaluation of American option prices in a path integral
framework using Fourier-Hermite series expansion\/}, Journal of Economic Dynamics and Control, 23 (1999), 1387-1424.

\bibitem{Davis2} Davis M.H.A., Panas V.G., Zariphopoulou T., {\em European Option Pricing with transaction costs\/}, SIAM Journal of Control and Optimization, 31 (1993), 470-493.

\bibitem{Frutos} de Frutos J., {\em A Spectral Method for bonds\/}, Computers and Operations Research, 35 (2008), 64-75.

\bibitem{Frutos2} de Frutos J., Garc\'ia-Archilla B., Novo J., {\em A postprocessed Galerkin method with Chebyshev or Legendre polynomials\/}, Numerische Mathematik, 86 (2000), 419-442.

\bibitem{FrutosGaton1} de Frutos J., Gat\'{o}n V., {\em A spectral method for an Optimal Investment problem with transaction costs under Potential Utility\/}, Journal of Computational and Applied Mathematics, 319 (2017), 262-276.

\bibitem{FrutosGaton2} de Frutos J., Gat\'{o}n V., {\em Chebyshev reduced basis function applied to option valuation\/}, Computational Management Science, 14 (2017), 465-491.


\bibitem{Karatzas2} Karatzas I. {\em Optimisation problems in the theory of continuous trading\/}, SIAM J. Control Optim., 27 (1989), pp. 1221-1259.

\bibitem{Magill} Magill M.J.P., Constatinides G.M., {\em Portfolio selection with transaction costs\/}, Journal of Economic Theory, 13 (1976), 245-263.

\bibitem{Soner} Soner H.M., Shreve S.E., Cvitani\'{c} J. {\em There is no nontrivial hedging portfolio for option pricing with transaction costs\/}, The Annals of Applied Probability 5 (1995), 2, 327-355.

\bibitem{Oosterlee} Zhang, B. and Oosterlee, C. W.{\em Pricing of early-exercise Asian options under L\'{e}vy processes based on Fourier cosine expansions\/}, Appl. Numer. Math., 78 (2014), 14-30.

\bibitem{Zhu} Zhu H. {\em Characacterization of variational inequalities in singular control\/}, Ph.D. thesis (1991), Brown University, Providence, RI.


\end{thebibliography}
\end{document}